\documentclass[11pt, reqno]{amsart}

\usepackage{amsmath, amsthm, amssymb, bbm, graphicx, epsfig, amsfonts}
\usepackage[usenames,dvipsnames]{xcolor}

\usepackage{tikz-cd}
\usepackage{mathtools}

\usepackage{graphicx}
\usepackage{listings}

\usepackage[normalem]{ulem}

\usepackage[margin=1in]{geometry}
\usepackage{lstautogobble}
\usepackage{enumerate}
\usepackage[shortlabels]{enumitem}
\usepackage{thmtools}
\usepackage{thm-restate}
\usepackage{amsthm}
\usepackage{verbatim}
\usepackage{accents}
\usepackage{mathtools}
\usepackage{physics}

 \usepackage[colorlinks=true, allcolors=magenta,linkcolor=blue,citecolor=magenta]{hyperref}
\usepackage[capitalise]{cleveref}  

\usepackage[bottom]{footmisc}
\crefformat{equation}{(#2#1#3)}

\usepackage{float}
\restylefloat{table}

\usepackage{mathrsfs}
\setlist{  
  listparindent=\parindent,
  parsep=0pt,
}

\theoremstyle{plain}

\newtheorem{thm}{Theorem}[section]
\newtheorem{prop}[thm]{Proposition}
\newtheorem{lemma}[thm]{Lemma}
\newtheorem{cor}[thm]{Corollary}

\theoremstyle{definition}
\newtheorem{defn}[thm]{Definition}
\newtheorem{remark}[defn]{Remark}

\numberwithin{equation}{section}

\Crefname{thm}{Theorem}{Theorems}
\Crefname{prop}{Proposition}{Propositions}

\numberwithin{equation}{section} 

\newcommand{\indic}{\mathbf{1}}
\newcommand{\p}{\partial}

\newcommand{\eps}{\varepsilon}

\newcommand{\Var}{\mathrm{Var}}
\newcommand{\cov}{\mathrm{Cov}}
\newcommand{\be}{\begin{equation}}
\newcommand{\ee}{\end{equation}}
\newcommand{\bea}{\begin{eqnarray}}
\newcommand{\eea}{\end{eqnarray}}
\newcommand{\beas}{\begin{eqnarray*}}
\newcommand{\eeas}{\end{eqnarray*}}

\newcommand{\inner}[2]{\langle #1,  #2 \rangle}

\newcommand{\bT}{\mathbb{T}}
\newcommand{\bS}{\mathbb{S}}
\newcommand{\bR}{\mathbb{R}}
\newcommand{\bN}{\mathbb{N}}

\newcommand{\bE}{\mathbb{E}}

\newcommand{\bZ}{\mathbb{Z}}

\newcommand{\cF}{\mathcal{F}}
\newcommand{\cP}{\mathcal{P}}

\newcommand{\fr}{\mathbf{r}}

\newcommand{\fC}{\mathbf{C}}

\newcommand{\bp}{\begin{proof}}
\newcommand{\ep}{\end{proof}}

\newcommand{\quni}{q_{\operatorname{unif}}}

\newcommand{\mhr}[1]{{\color{teal}{*#1*}}}

\definecolor{Orange}{HTML}{D55E00}
\newcommand{\km}[1]{{\color{Orange}{*#1*}}}

\begin{document}
\title[Phase Transitions and Linear Stability for the Kuramoto-Daido model]{Phase Transitions and Linear Stability for the mean-field Kuramoto-Daido model}

\author[K. Mun and M. Rosenzweig]{Kyunghoo Mun and Matthew Rosenzweig}
\address{Department of Mathematical Sciences, Carnegie Mellon University, Pittsburgh.}
\email{kmun@andrew.cmu.edu, mrosenz2@andrew.cmu.edu}

\thanks{M.R. was supported in part by NSF Grant DMS-2441170, DMS-2345533 and DMS-2342349.}


\begin{abstract}
    We consider the mean-field noisy Kuramoto-Daido model, which is a McKean-Vlasov equation on the circle with bimodal interaction $W(\theta)=\cos\theta+m\cos2\theta$ for $m\ge 0$ and interaction strength $K$, generalizing the celebrated noisy Kuramoto model corresponding to $m=0$. 

      Our first contribution is to characterize the phase transition threshold $K_{c}$ by comparing it to the linear stability threshold $K_\# = \min (1, m^{-1})$ of the uniform distribution. When $m \leq 1/2,$ $K_{c}=1$, coinciding with that of the Kuramoto model. On the other hand, for $m \geq 2$, we show $K_c= m^{-1}$. We also classify the regimes in which the phase transition is continuous or discontinuous. 
      

    Our second contribution is to analyze the linear stability of a global minimizer $q$ (the ``ordered phase'') of the mean-field free energy in the supercritical regime $K>1$. This stationary solution of the Kuramoto-Daido equation is unique up to translation invariance and distinct from the uniform distribution (the ``disordered phase''). Our approach extends the Dirichlet form method of Bertini et al. \cite{bertini} from the unimodal to bimodal setting. In particular, for $m \leq 1.590 \times 10^{-4}$ and $K>1$, we show an explicit lower bound on the spectral gap of the linearized McKean-Vlasov operator at $q$. To our knowledge, this is the first rigorous stability analysis for this class of models with bimodal interactions.
\end{abstract}

\keywords{}

\maketitle


\section{Introduction}
\subsection{Background}
We are interested in one-dimensional {McKean-Vlasov equations} set on the circle $\bT$, identified with $[0,2\pi]$ having periodic boundary conditions, of the form
\begin{align} \label{MV equation}
\partial_{t} q_{t} = {1 \over 2} \partial_{\theta}^{2} q_{t}- K \partial_{\theta}  \big(q_t \ \partial_{\theta}(W * q_{t}) \big), \quad  \theta \in \bT, \ t \geq 0,
\end{align}
Here, $W$ is an interaction potential to be specified, and $K>0$ is the interaction strength. {The case $K=0$ corresponds to the linear heat equation, which is trivial and therefore excluded.}

Equation \eqref{MV equation} can be interpreted as a \emph{mean-field limit} for a system of interacting ``particles.'' At the microscopic level, the dynamics are given by the system of SDEs
\begin{align}\label{Dynamics}
d\theta_{i}=  {K \over N} \sum_{j=1}^{N}  W'(\theta_i - \theta_j) \ dt + d B_{i}, \qquad  i\in [N],
\end{align}
where $N$ is the number of particles and $(B_i)_{i=1}^{N}$ are iid Brownian motions on $\bT$. The total ``force'' experienced by a single particle, which is formally $O(1)$, is proportional to the average of the pairwise forces generated by all the other particles in the system,{\footnote{We include the self-interaction term $i=j$ for convenience. This is harmless as it is an $O(1/N)$ constant and has no change on the limiting dynamics.}} hence the terminology mean-field. The mean-field limit refers to the convergence as $N \to \infty$ of the {\it empirical measure} 
	\begin{equation}\label{eq:EMt}
		q_{N,t}\coloneqq \frac1N \sum_{i=1}^N \delta_{\theta_i(t)}
	\end{equation}
associated to a solution $\Theta_N(t) \coloneqq (\theta_1(t), \dots, \theta_N(t))$ of the system \eqref{Dynamics}. Assuming the initial points $\Theta_N(0)$ are such that $q_{N,0}$ converges to a sufficiently regular measure $q_{0}$, a formal calculation leads one to expect that for $t>0$, $\nu_{N,t}$ converges to the solution $q_t$ of the macroscopic \emph{mean-field equation} \eqref{MV equation}.\footnote{Here and throughout this paper, we abuse notation by using the same symbol for both a measure and its Lebesgue density (when absolutely continuous).} For a general review of mean-field limits and the closely related notion of propagation of molecular chaos, we refer to the lecture notes and surveys \cite{jabin_review_2014,jabin_mean_2017,golse_dynamics_2016,golse_mean-field_2022,chaintron_propagation_2022,serfaty_lectures_2024}.



Equation \eqref{MV equation} can also be interpreted as a $2$-Wasserstein gradient flow,
    \begin{align}
    \partial_{t} q_t = - \partial_{\theta} \left[ q_t \partial_{\theta} \Big( \frac{\delta \cF}{\delta q}(q_t)  \Big)\right].
    \end{align}
where $\mathcal{F}$\footnote{For convenience, we consider relative entropy with respect to the uniform distribution $(2\pi)^{-1}$ instead of the usual Boltzmann-Shannon entropy.} is  \begin{align} \label{original free energy def} 
    \mathcal{F}(q)= {1 \over 2}  \int q \log ({q \over (2\pi)^{-1}}) \  d\theta - {K \over 2} \int q (W* q) \ d\theta.  
\end{align} 
Remark that $\mathcal{F}$ is the $N\rightarrow\infty$ limit of the free energy associated to the microscopic system \eqref{Dynamics}, and consequently, we refer to it as the \emph{mean-field free energy}. A function $q$ is a critical point of the free energy $\cF$ if and only if it is a stationary solution of McKean-Vlasov equation. In particular, the global minimizer of $\cF$ is significant as it represents a stable equilibrium state toward which the system naturally tends. 

For an attractive interaction, there exists a competition between the interaction potential and entropy, which is repulsive, depending by $K>0.$ If $K$ is sufficiently small, then the entropy dominates the interaction term in the free energy, and the minimizer of the free energy is the uniform distribution $\quni := (2\pi)^{-1}.$ And as $K$ gets bigger, a global minimizer is necessarily non-uniform. This change in behavior corresponds to a \emph{phase transition}. Although phase transition is an overloaded and often loosely used term in physics and mathematics, it can be made precise in this context. Following \cite{chayes,longtime},  a phase transition occurs when there is a \emph{critical interaction strength} $K_{c}>0$ such that \begin{enumerate} \label{phase transition def}
    \item in the \emph{subcritical} regime $K< K_{c},$  $q_{\operatorname{unif}}$ is the unique minimizer of free energy $\cF$;
    \item in the \emph{critical} regime $K=K_{c},$ $q_{\operatorname{unif}}$ remains a minimizer of $\cF$ {(though a priori not necessarily unique)};
    \item in the \emph{supercritical} regime $K>K_{c},$ there exists a non-uniform minimizer $q \neq q_{\operatorname{unif}}$ of $\cF.$
\end{enumerate}
In fact, the minimum of $\mathcal{F}$ over all probability measures is a decreasing function of $K$ for the class of interactions we will consider. Studying the number and {the second variation of} critical points of $\cF$ allows us to determine the multiplicities of stationary solutions of the McKean-Vlasov equation.  


In this paper, we assume an interaction potential $W$ of the form
\begin{align}  \label{W first assumption}
    W(\theta) = \sum_{n=1}^{\infty} c_n \cos n \theta,
\end{align}
where $(c_n)_{n \geq 1}$ are nonnegative real numbers in $l^{2}(\bZ).$ We assume that at least one Fourier coefficient is strictly positive, and without loss of generality, we set $c_1>0.$
As shown in \cite{vanderWaals}, the following two conditions are equivalent: \begin{enumerate}
    \item The model exhibits a phase transition at some $K_{c}.$
    \item There exists at least one $n \in \bN$ such that $c_n >0.$
\end{enumerate} 

Beyond the static analysis of the free energy, we are also interested in its \textit{dynamical} properties, particularly the stability of solutions. When and at what rate does a solution of \eqref{MV equation} converge to a stationary solution as $t \to \infty$? To understand this question, one can choose  an initial state $q_0$ near a stationary  $q$ and analyze the long-time asymptotics of $q_t.$ The difference $u_{t} = q_{t} - q$ satisfies \begin{align} \label{difference MV}
    \partial_{t} u_{t} = L_{q} u_{t} + f(u_{t}),
\end{align}
where $L_{q}$ is the linear operator defined by \begin{align} \label{linear MV operator}
    L_{q} u \coloneqq {1 \over 2} \partial_{\theta}^{2}u - K \partial_{\theta}\Big( q (\partial_{\theta}W*u) + u (\partial_{\theta}W *q) \Big),
\end{align} and $f(u) = -K \partial_{\theta}\big( u (\partial_{\theta}W * u) \big)$. {Remark that by conservation of mass, $u_t$ necessarily has zero average. One also quickly checks that $L_q q'=0$, corresponding to the translation symmetry.}

The works \cite{chayes, longtime} studied $K_{c}$ {by comparison with the} linear stability of the uniform distribution. When $q=\quni$, \eqref{linear MV operator} simplifies to \begin{align}
    L_{\quni}u= {1 \over 2} \partial_{\theta}^{2}u- K q_{\operatorname{unif}}(\partial_{\theta}^{2}W * u).
\end{align}
Evidently, $L_{q_{\operatorname{unif}}}$ is a Fourier multiplier, diagonalized by the Fourier basis $e_n \coloneqq e^{in \theta}$ with the corresponding eigenvalues given by \begin{align}
    \lambda_{n} \coloneqq {n^{2} \over 2} (Kc_{n}-1), \quad \forall n \geq 1.
\end{align}
The \emph{linear stability threshold} for $q_{\operatorname{unif}}$ is then determined by \begin{align} \label{K sharp def}
    K_{\#}\coloneqq ( \max_{n \geq 1} c_{n} )^{-1}.
\end{align}
An important problem is determining the conditions on the interaction $W$ under which $K_{c} = K_{\#}.$ It turns out that the solution to this problem is closely related to the \textit{continuity of the phase transition}.  \cite{chayes,longtime} investigated this problem for a broad class of interaction potentials. But to the best of our knowledge, a complete characterization of when $K_{c} = K_{\#}$ remains open for multimodal interaction potentials $W$. We refer to our forthcoming work \cite{mun_phase_nodate} for a characterization of this equivalence under some growth conditions on the Fourier coefficients of the interaction.

\subsection{Related studies}\label{ssec:intrors}

{Despite its (deceptively) simple mathematical appearance, equation \eqref{MV equation} encompasses a number of important models that we briefly mention before specializing to the Kuramoto-Daido model on which this paper focuses.}

A recent example is the \emph{(noisy) transformer model.} Despite the empirical success of large language models, the mathematical underpinnings of the underlying transformer architectures \cite{vaswani_attention_2017} are not fully understood. Geshkovski et al. \cite{transformer} proposed a surrogate model of self-attention mechanisms, {interpreting them as a mean-field interacting particle system on a (possibly high-dimensional) sphere ${\bS}^d$. In the special $d=1$ case, the setting of the model reduces to $\bT$ with the interaction potential}
\begin{align} \label{transformer interaction}
    W(\theta) =  {1 \over \beta} \exp(\beta \cos \theta) = {I_{0}(\beta) \over \beta} + {2 \over \beta} \sum_{n=1}^{\infty} I_{n}(\beta) \cos n\theta, 
\end{align}
where $I_{n}(\cdot)$ denotes the $n$-th modified Bessel function of the first kind and $\beta >0$ is a model parameter, which the authors interpret as inverse temperature, drawing an analogy to statistical mechanics. 
\cite{transformertry} further analyzed the noisy transformer dynamics associated to \eqref{transformer interaction}, along with its higher-dimensional counterparts. They computed $K_{\#}$ and identified bifurcations at $q_{\operatorname{unif}}$ for specific values of $K$, showing that the first bifurcation occurs at $K_{\#}$ via a linear stability analysis. {Recently, Balasubramanian et al. \cite{rigollet25} showed that there exist $\beta_{+} > \beta_{-} >0$ such that the phase transition is continuous at $K_{c} = \beta / (2I_{1}(\beta))$ for $\beta < \beta_{-},$ and is discontinuous at $K_{c} < \beta / (2I_{1}(\beta))$ for $\beta > \beta_{+}$.} 


Another example arises in crystallography in the form of the \emph{Onsager model} \cite{onsagerold}  to describe lyotropic liquid crystals. In this context, the interaction potential is
\begin{align}
    W(\theta) = -|\sin \theta| = -{2 \over \pi} + {4 \over \pi} \sum_{n=1}^{\infty} {1 \over 4n^2 -1} \cos 2n\theta.
\end{align} 
\cite{symmetricsolutionsonsager} showed that all stationary solutions {$q$} must be {axially symmetric $i.e.$, $q(\theta_0 - \theta) = q(\theta_0+\theta)$ for some $\theta_0$}, and for $K<0,$ infinitely many such solutions exist. \cite{upperKonsager} provided a lower bound for the critical interaction strength for a general class of {Lipschitz} interaction potentials $W$, showing that $
    K_{c} \geq (2\sum_{n=1}^{\infty} c_{n} )^{-1},$ which yields $K_{c} \geq {\pi/4}$ for the Onsager model. \cite{upperonsager2} refined this estimate to $K_{c} \geq \sqrt{\pi}/2$. In the forthcoming work \cite{mun_phase_nodate}, we prove that $K_c=3\pi/4$, coinciding with $K_{\#}$.
    

  In polymer dynamics, the \emph{Maier-Saupe model} \cite{maiersaupe} is governed by the interaction potential 
  \begin{align}
  W(\theta) = - \sin^{2}\theta = -{1 \over 2} + {1 \over 2} \cos 2\theta.
  \end{align}
 \cite{constantin2004note,luo_structure_2004} analyzed a phase transition in this model and classified stationary solutions. \cite{fatkullin2005note} further investigated a more general class of the form $W(\theta) = \cos m \theta$ for $m \in \bN$. 

The preceding models are unimodal, making their analysis structurally similar to the \emph{Kuramoto model} \cite{kuramoto} (also called the plane rotator model) for collective synchronization of oscillators, which is specified by the interaction potential
\begin{align} \label{kuramoto interaction}
      W(\theta) = \cos \theta.
  \end{align}
This is the case most extensively studied, and the one most relevant to our present work. Its unimodality permits an alternative interpretation as the \emph{$O(2)$/classical (mean-field) XY spin model}. \cite{messerspohn} established the mean-field limit and characterized the phase transition for this model, showing that  $K_c=1=K_{\#}$. Furthermore, in the supercritical regime $K>1,$ there are two (modulo translation) stationary solutions: the unstable $\quni$ and the global minimizer $q,$ of the free energy. Subsequently, \cite{bertini} analyzed the linear stability of $q$, and \cite{nonlinearstab} showed its nonlinear stability, as well as exponentially fast trend to equilibrium modulo phase shift. The latter work also characterized the (un)stable manifold of $\quni$. See also \cite{delarue2021uniform} for some later refinements.

Apart from our setting \eqref{Dynamics}, other works (e.g. \cite{ha2016collective, ha2020asymptotic}) studied a generalization of the Kuramoto model which an intrinsic natural frequency is considered. Each oscillator $\theta_{i}$ evolves with a time-independent random variable $\Omega_{i}$ drawn from a fixed distribution $g=g(\Omega)$,
\begin{align}
         {d\theta_{i} \over dt} = \Omega_{i} - {K \over N} \sum_{j=1}^{N} \sin(\theta_{i} - \theta_{j}), \quad \Omega_{i} \sim g, \qquad \forall i \in [N].
     \end{align} 
Depending on the choice of $g$ and the interaction strength $K,$ the system can exhibit various collective behaviors, such as synchronization and phase locking.
\medskip

\subsection{Kuramoto-Daido}\label{ssec:introKD}
For the proposed class of multimodal interaction potentials, the precise critical interaction strength $K_{c}$ and the multiplicity of stationary solutions remain largely unknown. Furthermore, the stability of these solutions has not yet been {rigorously} investigated. {As a step towards resolving this question, we consider bimodal interactions}.

This leads us to the \textit{Kuramoto-Daido model}, first introduced by Daido \cite{daidoold, daido1996multibranch}, featuring the interaction potential \begin{align} \label{kur-daido model}
    W(\theta)= \cos \theta + m  \cos 2\theta, \quad m\geq0.
\end{align}
Daido's motivation was to understand multibranch synchronization phenomena, which refers to the coexistence of multiple self-consistent synchronized states having distinct phase-locking structure and involves interactions with multiple Fourier modes. Through simulations and heuristics, he demonstrated phase transitions in such models. See \cite{daidophysics1, daidophysics2, daidophysics3} and references therein for further study in the physics literature. {For $m <0$, this system corresponds to the \emph{Hodgkin-Huxley model} \cite{hodgkin} {in neuroscience}, whose phase transitions have been analyzed in \cite{hodgkinphasetrans}.}




In the Kuramoto-Daido model, any stationary solution $q$ satisfies \begin{align}
    {1 \over 2} \big(\log q(\theta) \big)' + K \int \big( \sin(\theta - \varphi) + 2m \sin(2\theta- 2\varphi)\big) \ dq(\varphi) = 0.
\end{align}
{Unless stated otherwise, all integrals are understood to be over $\mathbb{T}$.} Denote the real and imaginary parts of the Fourier coefficients of $q$ by
\begin{align}\label{eq:rjsjdef}
    r_{j} \coloneqq \int \cos(j\theta) \ dq(\theta), \quad s_j \coloneqq \int \sin(j\theta) \ dq(\theta).
\end{align}
By the translation invariance of \eqref{Dynamics}, we may assume without loss of generality that $s_2=0$. Furthermore, by applying a reflection transformation $\theta \mapsto {\pi / 2} - \theta$ and an appropriate translation $\theta \mapsto \theta - \varphi$, we obtain $s_1=0$ (see \cref{parameter reduction section} for details). Thus, the stationary solution $q$ can be expressed in terms of $r_1$ and $r_2$ by the relation
\begin{align} \label{q def}
    q_{r_1, r_2}(\theta) = \frac{ \exp(2Kr_{1} \cos \theta + 2Kmr_{2} \cos 2\theta)}{\int \exp(2Kr_{1} \cos \theta + 2Kmr_{2} \cos 2\theta) \ d\theta}.
\end{align}
Integrating both sides against $\cos(j\theta)$, one sees that the parameters $r_j$ must satisfy the following  self-consistency (nonlinear) equations
\begin{equation}\label{sc equation}
    \begin{aligned}
         &r_{j} = \frac{ \int \cos(j\theta) \exp(2Kr_{1} \cos\theta + 2Kmr_{2} \cos2\theta) \ d\theta}{ \int \exp(2Kr_{1} \cos\theta + 2Kmr_{2} \cos2\theta) \ d\theta}, \qquad j=1,2.
    \end{aligned}
\end{equation}
{Retracing our steps, one sees that any solution $(r_1,r_2)$ yields a stationary solution $q_{r_1,r_2}$.}

When $m=0$, critical points are parametrized by a single quantity $\bar{r}_1 \in [-1,1]$ and the self-consistency equation \eqref{sc equation} collapses to\footnote{While \cite{bertini} used the notation $r$, we instead use $\bar{r}_{1}$.} 
\begin{align} \label{Kuramoto sc}
    \bar{r}_{1} = \frac{ I_{1}(2 K \bar{r}_{1})}{I_{0}(2 K \bar{r}_{1})},
\end{align} 
where $I_n$ is the modified Bessel function of the first kind. Bertini et al. \cite{bertini} exploited the unimodality of the dynamics to show a {continuous} phase transition at $K_{c}=1.$ 
Moreover, they classified solutions of \eqref{Kuramoto sc}, as presented in \cref{phase trans in Kura} below. It is elementary that $\bar{r}_1=0$ is unstable when $K>1$. As mentioned above, the minimizers $\pm \bar{r}_1$ are stable in both the linear and nonlinear senses, a point on which we elaborate more in the ensuing subsections.

\begin{thm}\cite{bertini} \label{phase trans in Kura}
    The solutions $\bar{r}_{1} \in [-1, 1]$ of the self-consistency equation \eqref{Kuramoto sc} are classified as follows. 
    \begin{enumerate}
        \item For $K \in (0, 1],$ there exists the unique solution $\bar{r}_{1}=0.$
        \item For $K >1,$ there are exactly three solutions $0, \  \pm\bar{r}_{1}(K)$, where $\bar{r}_{1}(K)$ is the unique solution in $(0, 1).$
    \end{enumerate}   
\end{thm}

Lucia and Vucadinovic \cite{Onsager} showed that for $m\in [0, {1/4}]$, the critical interaction strength remains $K_{c}=1$, identical to the Kuramoto model. Moreover, they used Morse index theory to classify stationary solutions determine their multiplicities in this small $m$-regime. We summarize their findings with the following theorem.

\begin{thm}\cite{Onsager}\label{exact multiplicities}
    Suppose $m \in (0, 1/4].$ Then, up to translation, the solutions of the self-consistency equation \eqref{sc equation} are classified as follows: \begin{enumerate}
        \item For $K \leq 1,$ there is a unique solution given by $q_{\operatorname{unif}} = (2\pi)^{-1}.$
        \item For $K \in (1, m^{-1}],$ there are two solutions $q_{\operatorname{unif}}, \ q_{r_1, r_2},$ where $(r_1, r_2)$ is the unique solution of \eqref{sc equation} in $(0, 1)^{2}.$
        \item For $K > m^{-1},$ there are three solutions: $q_{\operatorname{unif}}, \ q_{r_1, r_2}, \ q_{0, c},$ where $c\in(0, 1)$ is the unique solution of \begin{align}
            c = \frac{ \int \cos\theta \ \exp(2Km c \cos \theta) \ d\theta}{\int \exp(2Km c \cos \theta) \ d\theta}.
        \end{align}
    \end{enumerate}
\end{thm}

From now on, whenever $m \leq 1/4$, we let $r_1 = r_1(K, m),$ and $r_2 = r_2(K, m)$ denote the unique solution of the self-consistency equation \eqref{sc equation} given by \cref{exact multiplicities}. 

{We emphasize that in \Cref{phase trans in Kura,exact multiplicities}, the difficult part, particularly in the supercritical regime $K>1$, is not the existence, but rather the uniqueness/classification. The former follows from standard variational arguments, while the latter requires specialized analysis.}

\begin{remark}\label{rem:multidgen}
An important point that does not seem to have been explicitly highlighted in the literature is that the works \cite{messerspohn,bertini} for the Kuramoto and \cite{Onsager} for the Kuramoto-Daido models are not really one-dimensional results but rather \emph{unimodal/monochromatic} and \emph{bimodal/bichromatic} results, respectively. More precisely, if we replace $\bT$ by $\bT^d$ and $W(\theta)=\cos(\theta)$ by $W(\vec{\theta})=\cos(\vec{k}\cdot \vec{\theta})+m\cos(2\vec{k}\cdot\vec{\theta})$ for any fixed wave vector $\vec{k}\in \bZ^d$, then the proofs in the aforementioned works go through virtually unchanged.
\end{remark}

\subsection{Our contributions} \label{ssec:intro_contib}
{Having laid out the questions of interest concerning phase transitions and stability and motivated the choice of the Kuramoto-Daido model for addressing these questions, we now describe the main contributions of the present work.}

Our first contribution is to give a precise characterization of the critical interaction strength $K_{c}$ as a function of $m$, extending the analysis in \cite{Onsager}. If $2m \leq 1,$ then the critical interaction strength coincides with the unimodal case, recovering the Kuramoto result. However, when $2m>1$, the enhanced role of the $\cos2\theta$ mode lowers the critical interaction strength, leading to $K_{c}<1.$ We further compare $K_{c}$ with the linear stability threshold $K_{\#} = m^{-1}$ of the uniform distribution when $m > 1$. 

As part of this contribution, we also determine exactly when the phase transition is continuous or discontinuous. Following \cite{chayes, longtime}, a phase transition is said to be \textit{continuous} at the critical interaction strength $K_{c},$ if a global minimizer $q_{K_c}$ is unique and for any family of global minimizers $\{q_{K}: K > K_{c}\}$,\footnote{{The choice of topology in \eqref{eq:ptcont_def} is not important. It is a small exercise, which we leave to the reader, that if $q_K \rightharpoonup q$ in distribution as $K\rightarrow K_{c}^+$, then the convergence holds in any norm into which $C^\infty$ functions embed. This follows from the regularity of $q_K$. Furthermore, one only needs to consider the right-hand limit $K\rightarrow K_c^+$ because it is tautological that $q_K = \quni$ for $K<K_c$.}}
\begin{align}\label{eq:ptcont_def}
        \limsup_{K \to K_{c} +} \|q_{K} - q_{\operatorname{unif}} \|_{L^1} = 0.
\end{align}
Otherwise, the phase transition is \textit{discontinuous}.\footnote{Under very mild assumptions on $W$, which are satisfied in our case, the uniqueness of the global minimizer (modulo translation) at $K_c$ implies the condition convergence \eqref{eq:ptcont_def}.} 

{We summarize the above described results with the following theorem.}

\begin{thm} \label{phase theorem}
    There exists $m_{*} \in (1, 2)$ such that
    \begin{enumerate}
            \item\label{item:PT1} For $m \in [0, 1/2],$ $K_{c} = 1$ and the transition is continuous.
            \item\label{item:PT2} For $m \in (1/2, 1],$ $K_{c} < 1$ and the transition is discontinuous.
            \item\label{item:PT3} For $m \in (1, m_*),$ $K_{c} < m^{-1}$ and the transition is discontinuous.
            {\item\label{item:PT4} For $m = m_*,$ $K_{c} = m^{-1},$ but the transition is discontinuous.}
            \item\label{item:PT5} For $m \in (m_{*}, \infty) ,$ $K_{c} = m^{-1}$ and the transition is continuous.
        \end{enumerate}
        
\end{thm}

Remark that the monotonicity of the free energy $\cF_{K,m}$ in $K,m$ implies that $K_c(m)$ is non-increasing and continuous (see   \cref{monotonicity of K_c}). Furthermore, \cref{phase theorem} clarifies the complex behavior of $K_{c}(m)$ as depicted in the figure below.

{Unfortunately, we are currently unable to determine the exact value of $m_*$, only to give estimates for it. We refer to \Cref{ssec:introOP,ssec:PTm*} for further discussion on this point.}
    
    \begin{figure}[h]   
    \centering
    \includegraphics[width=0.7\textwidth]{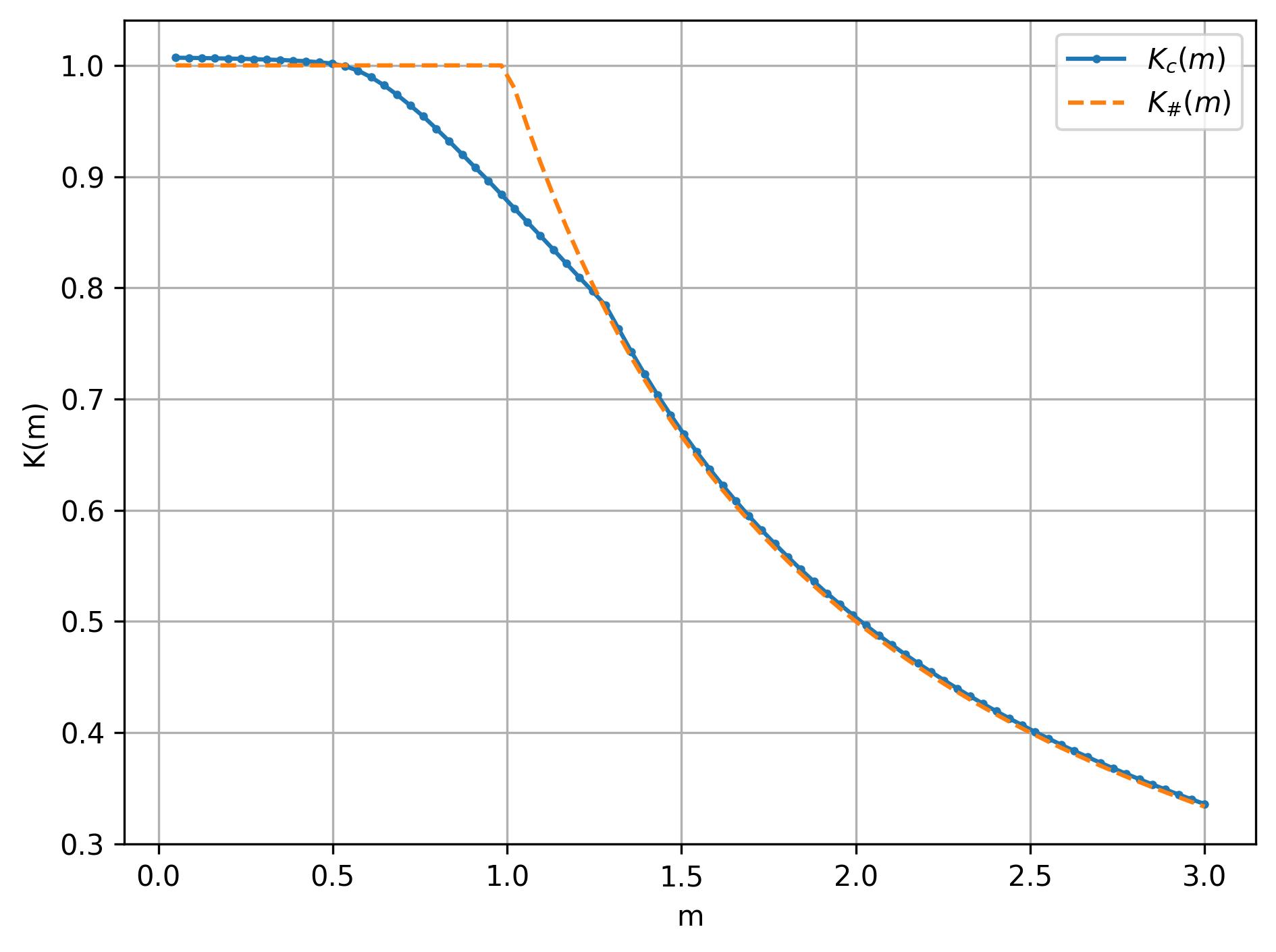}
    \caption{
    The \textit{blue} curve corresponds to the critical interaction strength $K_{c}(m),$ and the \textit{orange} curve corresponds to the linear stability threshold for the uniform distribution, given by $K_{\#}(m) = \min\{1, m^{-1}\}.$ Theorem \ref{phase theorem} gurantees that these two curves coincide for $m \in [0, 0.5] \cup [m_{*}, \infty)$, and that the \textit{blue} curve lies below the \textit{orange} curve in the remaining region $m \in (0.5, m_{*}).$}
    \label{fig:K_c-diagram}
    \end{figure}

\begin{remark}
Previous work \cite{chayes} identified a criterion relating the continuity of the phase transition to the comparison between $K_{c}$ and $K_{\#}$ (see \cref{phase trans continuity prop} below). Our \cref{phase theorem} provides a concrete example that is not covered by their criterion, showing that there exists $m_{*} \in (1, 2)$ such that $K_{c}(m_{*}) = K_{\#}(m_{*})$, but the phase transition is discontinuous. In particular, this demonstrates that the criterion in \cite{chayes} is not a dichotomy. 
\end{remark}


Our second---and most technically involved---contribution is to generalize the linear stability analysis of global minimizers of the Kuramoto model done by Bertini et al. \cite{bertini} to the Kuramoto-Daido model. Although there are some numerical investigations and formal stability analysis \cite{bertoli_stability_2024,bertoli_correction_2025}, to the best of our knowledge, our \cref{stability thrm} presented below is the first rigorous result for multi-modal interactions.

To state the result, let $L_0^2(\bT)$ denote the zero-average subspace of $L^2(\bT)$ and introduce the rigged Hilbert space
\begin{align}\label{eq:rigHS}
    H_q^1(\bT) \subset L_0^2(\bT) \subset H_{1/q}^{-1}(\bT),
\end{align}
where
\begin{align}\label{eq:Hq1def}
    H_q^{1} \coloneqq \{f\in L_0^2 : \|f\|_{H_q^1} < \infty\}, \qquad \|f\|_{H_q^1}^2 \coloneqq \int |\p_\theta f|^2 qd\theta,
\end{align}
and $H_{1/q}^{-1}$ is the dual with respect to the distributional pairing $\langle{\cdot,\cdot}\rangle$. Concretely,
\begin{align}\label{eq:Hq1'def}
    H_{1/q}^{-1} = \{f\in \mathcal{S}'(\bT) : \hat{f}(0) = 0 \ \text{and} \ \|f\|_{{H}_{1/q}^{-1}}<\infty\}, \qquad \|f\|_{{H}_{1/q}^{-1}}^2 \coloneqq \int |\p_\theta^{-1}f|^2\frac{d\theta}{q},
\end{align}
where $\partial_\theta^{-1}$ is the Fourier multiplier defined by $\widehat{\partial_\theta^{-1}f}(k) = (ik)^{-1}\hat{f}(k)\indic_{k\ne 0}$.\footnote{If $\int fd\theta = 0$, then it is straightforward that $\p_\theta^{-1}f = \int_0^{(\cdot)}fd\theta$.}
Remark that $H_q^1, H_{1/q}^{-1}$  and their norms are equivalent to the usual Sobolev spaces $H^1 = H_{\quni}^1$, $H^{-1} = H_{1/\quni}^{-1}$. This follows from the trivial bounds $e^{-4K\|W\|_{L^\infty}}\le q\le e^{4K\|W\|_{L^\infty}}$.

\begin{thm} \label{stability thrm}
For any $m\in\bR$, the operator $L_q$ densely defined on $H_{1/q}^{-1}$ with domain $D(L_q) = H_q^1$ is self-adjoint and has pure point spectrum $\sigma(L_q)$.

Define the threshold
\begin{align}\label{eq:m0_def}
    m_0 \coloneqq 1.590 \times 10^{-4}.
\end{align}
Assume $K>1$ and $m \in (0,m_0]$. Then $\sigma(L_q)\subset (-\infty, 0]$, and the eigenvalue $0$ has one-dimensional eigenspace spanned by $q'$. Moreover, we have the spectral gap
\begin{align}
    \mathrm{gap}(L_{q}) \geq {C_{1} \over C^2},
\end{align}
where $C_{1}>0$ is defined in  in \cref{Dirichlet lower} \eqref{eq:Dirichlet lower} and $C>0$ is the optimal constant of the Poincar\'{e} inequality $\|\cdot\|_{H_{1/q}^{-1}}\le C\|\cdot\|_{L_{1/q}^2}$. 
\end{thm}

{We emphasize that \cref{stability thrm} shows if $m\le m_0$, then a spectral gap holds for all $K>1$. It is not too difficult to prove using perturbation theory \`{a} la Kato \cite{kato2013perturbation} that for any $m$ sufficiently small, there exists a $K_m>1$ such that for any $1<K\le K_m$, a spectral gap holds. In fact, combining this perturbation theory with \cref{rem:multidgen}, one sees that in any dimension and for sufficiently regular $W$ with a dominant mode, there exists an $m_0>0$ such that for any $m\le m_0$, there is a $K_m>K_c$ such that for $K_c<K\le K_m$, a spectral gap holds. Since the argument is quantitative, the proof yields an explicit lower bound for the spectral gap. The difficulty lies in showing that a spectral gap holds for all $K>K_c$ once $m$ is sufficiently small (i.e. removing the upper bound $K_m$), and overcoming this difficulty requires starting the proof of the spectral gap from scratch. } 

\begin{remark}\label{rem:noisytransformer}
     Since the ``high-temperature'' limit of the noisy transformer model in \cite{transformertry,rigollet25} is the Kuramoto model, this perturbative approach described above yields a spectral gap for the noisy transformer model for small $\beta$, which seems to be a new observation.
\end{remark}

\subsection{Outline of the proof}
Let us briefly comment on the proofs of \Cref{phase theorem,stability thrm}.

The starting point of our analysis is to pass to a finite-dimensional problem by parametrizing critical points of the free energy \eqref{original free energy def} in terms of the real parts $r_{1},r_{2}$ of their first and second Fourier coefficients. The critical points of \eqref{original free energy def} are in bijection with the critical points of the two-dimensional free energy $F_{K,m}$ introduced in \eqref{F def}. Having made this passage, we are in a position to prove \cref{phase theorem} by separately analyzing phase transitions for the cases $m \leq 1/2$ and $m>1/2$, as we describe below. This is all presented in \cref{section 3}.

For $m \leq 1/2,$ we show that the phase transition occurs at the linear stability threshold $K_{\#}=1$ of the uniform distribution. Specifically, for $K\leq 1,$ we establish the uniqueness of the critical point of $F_{K, m}.$ This follows from the symmetry of $F_{K, m},$ the one-dimensionality of the manifold of critical points (see  \cref{search space shrink}), and convexity estimates of the free energy established in \cite{Onsager}. For $K>1,$ we apply a perturbation argument to show that there exists a small nontrivial solution $(r_1, r_2) \neq (0, 0)$ such that $F_{K, m}(r_1, r_2)<0,$ implying that the uniform distribution no longer minimizes $F_{K, m}.$ 

For $m > 1/2,$ we derive upper bounds on $K_{c}$ by carefully constructing $(r_{1},  r_{2})$ using perturbation arguments to ensure that the free energy is negative. Specifically, setting $(r_{1}, r_{2}) = (\eps(m)^{1/2}, \eps(m))$ for small $\eps(m)>0,$ we show that $K_{c} < 1- O((m-{1/ 2})^{2}).$ For $m \geq 2,$ we write $F_{K, m}$ in terms of the rescaled interaction strength $\tilde{K} \coloneqq K/m$ and show $K_{c} = m^{-1}$ by a similar argument to the one showing that $K_{c}=1$ when $m \leq 1/2.$ Furthermore, from the continuity of free energy, we deduce there exists $m_* \in (1, 2)$ such that $K_{c}<m^{-1}$ for $m\in (1, m_*)$ and $K_{c} = m^{-1}$  for $m\geq m_{*}.$

\medskip
Having described the proof of \cref{phase theorem}, we turn to the proof of \cref{stability thrm}, which is divided over \Cref{section 4,sec:Tech}. Since a detailed outline of the proof is presented in \cref{ssec:stab_out}, we shall be brief in our remarks here. 

At high level, our approach is inspired by that of \cite{bertini} for the Kuramoto model. But due to the dependence on two intertwined order parameters $r_1,r_2$, the analysis in the present paper is much more computationally involved and is not a straightforward extension of \cite{bertini}, nor completely analogous.  Recalling the Hilbert rigging \eqref{eq:rigHS},  our goal is to show a lower bound for the Dirichlet form
\begin{align}\label{eq:introDF}
    \mathbf{D}(u)= -\inner{u}{L_{q} u}_{H^{-1}_{1/q}} = {1 \over 2} {\inner{u}{u}_{L^{2}_{1/q}} - \inner{u} {q{J}*u}_{L_{1/q}^2}} 
\end{align}
after modding out the kernel of $L_q$, which turns out to be spanned by $q'$. As previously mentioned, this kernel is generated by the translation symmetry of the model. The diligent reader will recognize that the last expression in \eqref{eq:introDF} is the second variation of the free energy \eqref{original free energy def}. 

The primary challenge in estimating the Dirichlet form arises from the intertwining nature of the contributions of $\cos\theta$ and $\cos2\theta$, generating higher-order Fourier coefficients. Moreover, a number of quantities that appear in the proof of the lower bound are explicitly computable or take simple form when $m=0$, but not when $m>0$. To handle this problem, we rely on first-order approximation estimates in $m$ of the form $r_j = \bar{r}_j + O(m)$, for $j=1,2$, where $\bar{r}_{j}$ is the value of $r_{j}$ at $m = 0.$ From this approximation, we generalize Tur\'anian inequalities \cite{turanian, turanianmonotonicity} for $\bar{r}_{1}$ and $\bar{r}_{2}$ into similar bounds for $r_1$ and $r_2.$ These technical generalization enables us to estimate the complicated terms in the Dirichlet form. On the other hand, the perturbative nature of the argument is responsible for restriction on $m$ beyond $m\le {1/4}$.

\subsection{Further directions and open problems} \label{ssec:introOP}
In closing the introduction, we briefly indicate a few directions suggested by the present work and record several questions that remain open. We also highlight ongoing and forthcoming work that contains related developments.

The first question is the exact value of the unknown threshold $m_*$ from \cref{phase theorem}, which only states that $m_* \in (1,2)$. The numerics behind \cref{fig:K_c-diagram} suggest the shaper range
\begin{align}
    1.245 < m_{*} < 1.282.
\end{align}

Another important question is the number of non-uniform stationary solutions and their stability for $m > 1/4$ and $m > m_{0},$ respectively. In particular, when $m> m_{*},$ we conjecture that the model is identical to the case $m< 1/2,$ in terms of the number of stationary solutions and their  stability. The reason is that these two situations share a continuous phase transition, and the contribution of one cosine mode dominates the other. In the discontinuous phase transition region $m \in (1/2, m_{*}),$ the situation is more complicated. We do not know the exact value of $K_{c}$, nor a classification of the stationary solutions in the (super)critical regimes, let alone their stability. At present, we only have bound for $K_{c}$ when $m \geq 1/2$ (see \Cref{m near half upper,big m phase trans}),
\begin{align}
    {1 \over 2(1+m)} \leq K_{c} \leq \min \left\{ {1 \over m}, 1 - O\left({(m^{2} - 1/4)^{2} \over m^{8}}\right) \right \}.
\end{align}

Regarding linear stability, when $K>1$ and $m \leq m_{0},$ we do not expect the lower bound given by \cref{stability thrm} to be optimal. One may obtain a sharper bound than that in  (see \cite[Section 2.5]{bertini}), but this resulting bound is still expected to be suboptimal and so we omit it. In fact, even when $m=0$, the optimal spectral gap is unknown to our knowledge. The lower bound obtained in \cite{bertini} vanishes as $K\rightarrow\infty$, while the authors of that work suggest, based on numerics, that the gap does not vanish in the large $K$ limit. Finally, there is the obvious problem of showing the (non)existence of a spectral gap when $m>m_0$, the simplest case to consider being when $m_0 < m\le {1/4}$. 

Turning to nonlinear stability/trend to equilibrium (under the assumption, when $K<K_c$, it is easy to prove that starting from any initial datum, the solution $q_t$ of \eqref{MV equation} converges to the uniform measure $\quni$ exponentially fast at a rate determined by the spectral gap of $L_{\quni}$. In the critical case $K=K_c$, an exponential rate of convergence to $\quni$ fails, due to the vanishing of the spectral gap, but one can still obtain polynomial rates of convergence (cf. \cite{monmarche_local_2025}). The supercritical regime $K>K_c$ is more subtle. It is not hard to upgrade our linear stability result to a \emph{local}, exponentially fast convergence to the global minimizer (cf. \cite{nonlinearstab} for the Kuramoto model). The obtained rate is determined by the value of the spectral gap for $L_q$. In fact, we  know more. Because the free energy $\mathcal{F}_{K,m}(q_t)$ is decreasing along the flow, once it is below $\mathcal{F}_{K,m}(\quni)=0$ when $K\le m^{-1}$ and $\mathcal{F}_{K,m}(q_{0,c})<0$ when $K>m^{-1}$, it must remain below these respective thresholds for all future time. Since by soft arguments, $q_t$ converges to a stationary solution, it follows that there is some $T_*$ after which the convergence to a global minmizer is exponentially fast. The difficulty is determining precisely the conditions under which the free energy crosses these thresholds and after how long---what would constitute a global convergence result. Unlike the Kuramoto model whose global convergence is characterized by the (non)vanishing of the first Fourier mode \cite{nonlinearstab,delarue2021uniform}, no such simple characterization exists when $m>0$. Moreover, even though $\quni, q_{0,c}$ are unstable in a global sense, they are still stable under certain perturbations. For  example, any solution with zero first and second Fourier modes necessarily converges to $\quni$ exponentially fast. Thus, it is important to completely determine the (un)stable manifolds of $\quni, q_{0,c}$. These questions, which we have partially resolved, are the subject of ongoing work. Remark that the choice of topology in the above discussion is immaterial, since the parabolic smoothing of the flow allows for convergence in any Sobolev norm.



{Finally, let us take this opportunity to advertise some closely related work in preparation.} As remarked in \cref{ssec:intrors}, the energy of the Kuramoto model is isomorphic to that of the $O(2)$/classical mean-field XY model. It is a celebrated result of Ellis and Newman \cite{fluctuation} that the fluctuations of the total magnetization/order parameter $\sum_{j=1}^N e^{i\theta_j}$, when $\Theta_N=(\theta_j)$ are distributed according to the canonical Gibbs ensemble, are non-Gaussian in the $N\rightarrow\infty$ limit at the critical inverse temperature $\beta_c=2K_c=2$ and, moreover, scale like $N^{3/4}$ in contrast to the $N^{1/2}$ scaling for iid $\theta_i$ observed in the sub-and supercritical regimes. See also the later work \cite{eichelsbacher_steins_2010,chatterjee_nonnormal_2011,  kirkpatrick_asymptotics_2016, garoni2025rate} that obtains quantitative versions of their convergence result using Stein's method. It has been speculated (see, e.g., \cite[Section 2.2.1]{delgadino2023phase}) that this non-Gaussianity and anomalous scaling at the critical temperature are universal, but to the best of our knowledge, the Kuramoto model is the only example within the class of interactions \eqref{W first assumption} for which this has been shown. By revisiting \cite{fluctuation}, we show in  the forthcoming work \cite{mun_critical_nodate} non-Gaussianity and anomalous scaling of fluctuations in the critical regime $K=K_c$ of canonical ensembles with general multimodal energies. We go beyond \cite{fluctuation} by characterizing not just the limiting behavior of fluctuations of the order parameter but also those of the full empirical measure. Interestingly, the anomalous scaling only holds for the first order parameter; and while the first and second order parameters have non-Gaussian limits, the higher order parameters have Gaussian limits.  Building on the more recent work \cite{eichelsbacher_steins_2010, chatterjee_nonnormal_2011,kirkpatrick_asymptotics_2016, garoni2025rate}, we even have quantitative rates. To the best of our knowledge, this is the first such result for Gibbs ensembles associated to multimodal energies. 

Tying the discussion of the preceding paragraph back into the dynamics \eqref{Dynamics}, it would be interesting to observe the timescale on which the non-Gaussianity and anomalous scaling emerges in the critical regime. Starting from iid positions, the fluctuations of the total magnetization behave as in the usual CLT initially. But since the joint distribution of $\Theta_N(t)$ converges to the canonical ensemble as $t\rightarrow\infty$, the fluctuations should grow and non-Gaussianity should emerge on some $N$-dependent timescale. Investigating this question is a natural next step.

\subsection{Organization of the paper}
  In \cref{Section 2}, we recall some technical properties of Bessel functions that will be used throughout the body of the paper.
 In \cref{section 3}, we prove our first main result, \cref{phase theorem}. In \cref{section 4}, we prove our second main result, \cref{stability thrm}. The proof of the lower boundedness of the Dirichlet form behind \cref{stability thrm} relies on several auxiliary technical estimates that are deferred to \cref{sec:Tech}. At the beginnings of \Cref{section 3,section 4,sec:Tech}, we give further comments explaining the organization  of their respective subsections.
 

\subsection{Acknowledgments}
The authors thank Minh-Binh Tran for sharing useful references regarding the Kuramoto-Daido model. The authors also thank Cole Graham for helpful discussion.

\section{Bessel functions}\label{Section 2}
{In this section, we review some preliminary inequalities and properties of modified Bessel functions.}

From \eqref{sc equation}, we naturally introduce the following Bessel functions to investigate properties of $r_1$ and $r_2.$ We write  $I_{n}(x, y),$ the $n$-th Fourier coefficient of $\exp(x \cos \theta + y \cos 2\theta).$ \begin{align} \label{Bessel def}
    I_{n}(x, y) \coloneqq {1 \over 2\pi} \int \cos n \theta \ \exp( x \cos \theta + y \cos 2\theta) \ d\theta, \quad n \in \mathbb{Z}, \ x, y \in \bR.
\end{align}
\cite{Onsager} studied the convexity of $\log I_{0},$ leading to the following result.

\begin{lemma}\cite[Proposition 2.2, Lemma 7.4]{Onsager} \label{I technical lemmas}The following inequalities hold. \begin{enumerate}
    \item For all $n \in \mathbb{Z}$, we have \begin{align}
        I_{n}(x, y) >0, \quad x, y >0.
    \end{align}
    \item For $(x, y ) \in (0, \infty)^{2},$   \begin{align} \label{I property 1}
        \frac{2}{x^{2} I_{0}^{2}} \left(I_{1} (I_{0} + xI_{1}) - xI_{0}  \partial_{x}I_{1} \right) \geq \frac{1}{I_{0}^{2}} ( \partial_{xy}I_{0}  I_{0} - I_{1} I_{2}) >0.
    \end{align}
    \item For $(x, y ) \in (0, \infty)^{2},$  \begin{align} \label{I property 2}
        {I_{1} \over x I_{0}} - {1 \over 2I_{0}^{2}} (\partial_{yy} I_{0}  I_{0} - I_{2}^{2}) > {1 \over xI_{0}^{2}} ( \partial_{xy}I_{0}  I_{0} - I_{1} I_{2}).
    \end{align}  
\end{enumerate}   
\end{lemma}

Furthermore, we define the one-variable Bessel function, which is special case of $I_{n}(x, y)$ when $y=0.$ We denote it as \begin{align}
    {I}_{n}(x) \coloneqq {1 \over 2\pi} \int \cos n\theta \ \exp(x \cos \theta) \ d\theta, \quad n \in \mathbb{Z}, \ x \in \bR.
\end{align}
with abuse of notations to \eqref{Bessel def}. 
\begin{lemma}\cite[Section 3.7]{watson1922treatise} The following statements hold. 
    \begin{enumerate}
        \item \label{non-negativity of bar I} For $n \geq 0,$ 
        \begin{align}
        I_{n}(x) \geq 0 \quad \forall x \geq 0,
        \end{align}
        with the equality if and only if $x=0.$
        \item \label{monotonicity of I} For $x \geq 0,$ the function 
        \begin{align}
            \bZ_{ \geq 0} \to [0, \infty), \quad n \mapsto I_{n}(x) 
        \end{align}
        is monotonically decreasing.
    \end{enumerate}
\end{lemma}

 Estimates in Tur\'anian function \cite{turanian, baricz2010turan, baricz2013turan, joshi1991some, thiruvenkatachar1951inequalities}  enables to quantitatively estimate different Bessel functions.

\begin{lemma}\cite[Section 3.4, (2.9)]{turanianmonotonicity} \label{turanian lemma}
    Define the Tur\'anian function as \begin{align}
        T_{n}(x) \coloneqq I_{n}^{2}(x) - I_{n-1}(x) I_{n+1}(x), \quad n \in \bZ, \ x \in \bR.
    \end{align}
    The following properties hold.
    \begin{enumerate}
        \item $T_{n}(x) > 0$ for all $n>-1$ and $x>0.$
        \item\label{turanian 2} For $n \in \bN,$ \begin{align}
            T_{n}(x) \leq {1 \over n+1} I_{n}^{2}(x) \quad \forall x \in \bR.
        \end{align}
        \item\label{monotonicity of turanian} Fix $x>0.$ Then, the function $n \mapsto T_{n}(x)$ is monotonically decreasing on $[0, \infty).$
    \end{enumerate}
\end{lemma}

\section{Phase transitions in the Kuramoto-Daido model}\label{section 3}
In this section, we investigate phase transitions for the Kuramoto-Daido model. The main objective is the proof of \cref{phase theorem}.

Previous work \cite{chayes,delgadino2023phase} studied the continuity of phase transitions for general periodic McKean-Vlasov models, relating the critical interaction strength $K_c$ and the (dis)continuity of the phase transition to the linear stability threshold $K_{\#}$ for $\quni$. We will make use of the following criterion from \cite[Proposition 2.12]{chayes}.

\begin{prop} \label{phase trans continuity prop}
Suppose the model \eqref{MV equation} exhibits a phase transition at $K_{c},$ and let $K_{\#}$ be the linear stability threshold for the uniform distribution, as defined in \eqref{K sharp def}.
\begin{enumerate}
    \item If $q_{\operatorname{unif}}$ is the unique global minimizer of $\cF$ at $K=K_{\#},$ then $K_{c} = K_{\#}$ and the phase transition is continuous.
    \item If $q_{\operatorname{unif}}$ is not a global minimizer of $\cF$ at $K=K_{\#},$ then $K_{c} < K_{\#}$ and the phase transition is discontinuous.
\end{enumerate}
\end{prop}
Remark that this criterion does not provide a perfect dichotomy for determining the continuity of phase transitions since $\quni$ may still be a non-unique minimizer at $K=K_{\#}$. Nevertheless, the Kuramoto-Daido model with $m \in (0, 1]$ precisely divides into these two regimes at $m = 1/2$. We prove this in  \cref{phase trans classification} below. We note that in our model
\begin{align}
    K_{\#}= \begin{cases}
        1, & \text{for } m \in (0, 1],\\
        m^{-1}, & \text{for } m >1.
    \end{cases}
\end{align}

\subsection{Parameter Reduction} \label{parameter reduction section} 
    
We reduce the number of Fourier coefficients necessary to describe the solution $q$. The reader will recall from \eqref{eq:rjsjdef} the notation $r_j,s_j$ the Fourier coefficients of $q$. As in \cref{ssec:introKD}, we use translation invariance to assume $s_2=0$. Then $(r_1,r_2,s_1)$ satisfy the self-consistency system 
\begin{equation} \label{original sc equation}
    \begin{aligned}
         &r_{j} = \frac{ \int \cos(j\theta)  \exp(2Kr_{1} \cos\theta + 2Kmr_{2} \cos2\theta+2Ks \sin \theta) \ d\theta}{ \int \exp(2Kr_{1} \cos\theta + 2Kmr_{2} \cos2\theta+2Ks_1 \sin \theta) \ d\theta}, \\
        &s_1 = \frac{ \int \sin\theta  \exp(2Kr_{1} \cos\theta + 2Kmr_{2} \cos2\theta+2Ks_1 \sin \theta) \ d\theta}{ \int \exp(2Kr_{1} \cos\theta + 2Kmr_{2} \cos2\theta+2Ks_1 \sin \theta) \ d\theta}.
    \end{aligned}
\end{equation} 
By applying several basic transformations, we show that stationary solutions can be parameterized by $r_1$ and $r_2$ while imposing $s_1=0$.

\begin{lemma} \label{s=0 lemma}
    All stationary solutions of \eqref{MV equation} can be written as either $q_{r_1, r_2}(\pi/2 -\theta)$ or $q_{r_1, r_2}(\theta - \varphi)$, where $(r_{1}, r_{2}) \in [-1, 1]^{2}$ satisfies the self-consistency equations \eqref{sc equation}.
\end{lemma}
\begin{proof}
Suppose $(r_1, r_2, s_1)$ solves \eqref{original sc equation}. Set $M \coloneqq \sqrt{r_1^{2} + s_1^{2}}$ and write $(r_1,s_1) = (M \cos \varphi, M\sin\varphi)$ for $\varphi \in [0, 2\pi)$. With this notation, we rewrite
    \begin{equation}
    \begin{aligned}
        &r_1 \cos \theta + s_1 \sin \theta = M( \cos \theta \cos \varphi + \sin \theta \sin \varphi) = M \cos(\theta- \varphi),\\
        &r_1 \sin\varphi - s_1 \cos \varphi = M (\cos\varphi \sin \varphi - \sin \varphi \cos \varphi) = 0.
    \end{aligned}
    \end{equation}
    From \eqref{original sc equation}, we have
    \begin{multline}
        r_{1} \sin \varphi \int \exp(2KM \cos(\theta-\varphi) + 2Km r_2 \cos2\theta) \ d\theta \\
        =  \int \sin \varphi \cos \theta \exp(2KM \cos(\theta-\varphi) + 2Km r_2 \cos2\theta) \ d\theta,
    \end{multline}
    \begin{multline}
        s_1 \cos \varphi \int \exp(2KM \cos(\theta-\varphi) + 2Km r_2 \cos2\theta) \ d\theta\\
        =  \int \cos \varphi \sin \theta \exp(2KM \cos(\theta-\varphi) + 2Km r_2 \cos2\theta) \ d\theta.
    \end{multline}
    Subtracting these two equations, we obtain
    \begin{align}
        I(M , \varphi)\coloneqq \int \sin(\theta- \varphi)  \exp(2KM \cos(\theta-\varphi) + 2Km r_2 \cos2\theta) \ d\theta = 0.
    \end{align}
    Changing variables $\theta \mapsto \theta - \varphi$, \begin{align}
        I(M , \varphi)= \int \sin\theta  \exp(2KM \cos\theta)  \exp( 2Km r_2 \cos(2\theta +2\varphi) ) \ d\theta.
    \end{align}
    
    Next, partition $[-\pi,\pi) = \bigcup_{i=-1}^2 [(i-1)\frac\pi2,i\frac\pi2]$ and decompose $I(M, \varphi) = \sum_{i=-1}^{2} I_{i}(M, \varphi)$. Making the changes of variable $\theta \to \theta+ \pi$ in $I_{-1}$, $\theta \to -\theta$ in $I_0$, and $\theta \to \pi - \theta$ in $I_2$, we obtain 
   \begin{multline}
       I(M, \varphi) = \int_{0}^{\pi /2} \sin \theta \ \big( \exp(2KM \cos \theta) - \exp(-2KM \cos \theta) \big) \\
         \times\big( \exp(2Kmr_{2} \cos( 2\theta + 2\varphi)) - \exp(2Kmr_{2} \cos(2\theta - 2 \varphi)) \big) \ d\theta. 
   \end{multline}
    Assuming $M >0$ and $r_2 \neq 0$, we now analyze the sign of $I(M, \varphi)$.
    
    If $\varphi \in (0, {\pi / 2})$, then for every $ \theta \in (0, {\pi / 2})$, $\cos \theta > - \cos \theta$ and $\cos(2\theta + 2\varphi) < \cos(2\theta - 2\varphi)$. Hence,
    \begin{align}
        I(M, \varphi) <0 \ \text{ if } \ r_2 >0,  \quad \text{and} \quad I(M, \varphi)>0 \ \text{ if } \ r_2 <0.
    \end{align}
    By similar arguments for $\varphi \in ({\pi / 2}, \pi) \cup  ({\pi}, 3\pi/2) \cup  ({3\pi / 2}, 2\pi)$, we find that $I(M, \varphi) \neq 0$ unless {at least one of the following holds:}
    \begin{align}
        M=0, \quad r_2 = 0, \quad \varphi \in \{0, {\pi / 2}, \pi, {3\pi / 2} \}.
    \end{align}
    Consider each case:
    \begin{enumerate}
        \item If $M=0$, then $r_1 = s_1=0$. 
        \item If $r_2 =0$, then 
        \begin{align}
            q(\theta) = \frac{ \exp(2KM \cos (\theta - \varphi))}{ \int \exp(2KM \cos (\theta - \varphi)) \ d\theta} \implies  q(\theta+\varphi) = \frac{ \exp(2KM \cos \theta)}{ \int \exp(2KM \cos \theta ) \ d\theta}.
        \end{align}
        Thus, $q$ is a translate of $q_{M, 0}$.
        \item If $\varphi \in \{0, \pi\},$ then $s_1=M\sin \varphi = 0$.
        \item If $\varphi \in \{\pi/2, 3\pi/2\},$ then $r_1 =0,$ and
        \begin{align}
             q(\pi/2 -\theta) =  \frac{ \exp(2Ks_1 \cos \theta + 2Km(-r_{2}) \cos2\theta) }{ \int \exp(2Ks_1 \cos \theta + 2Km(-r_{2}) \cos2\theta) \ d\theta},
        \end{align}
        where $q(\pi/2 -\theta) = q_{s_1, -r_{2}}(\theta).$
    \end{enumerate}
    In all cases, $q(\theta + \varphi),$ $q(\pi/2 - \theta),$ or $q$ itself can be expressed as $q_{r_1, r_2}$, for some $(r_1, r_2)$ solving \eqref{sc equation}, as claimed. 
\end{proof}

\medskip

From now on, we assume $s_1=0$ and $(r_1,r_2)$ solves the two-dimensional self-consistency equation \eqref{sc equation}. Under this assumption, the free energy \eqref{original free energy def} simplifies to
\begin{align}
    {\mathcal{F}}(q) &= \int q \log ({q \over \quni}) \ d\theta  - K \iint \big( \cos(\theta-\varphi) + m \cos(2\theta-2\varphi) \big) \ dq(\theta) dq(\varphi) \nonumber\\
    &=\int q \log q \ d\theta + \log(2\pi) - K r_1^{2} - Km r_2^{2}.
\end{align}
Substituting $q(\theta) = {Z^{-1}} {\exp(2Kr_1 \cos \theta+ 2Kmr_2 \cos2\theta)}$, where $Z$ is the normalizing constant, into $\mathcal{F}$, we obtain
\begin{align}
    \cF(q) &= 2K r_1 \int \cos \theta \ dq(\theta) + 2Km r_2 \int \cos2\theta \ dq(\theta) - \log ({Z \over 2\pi}) - Kr_1^{2} - Km r_2^{2} \nonumber\\
    &= Kr_1^{2} + Kmr_2^{2} - \log \Big({1 \over 2\pi}\int \exp(2Kr_1 \cos\theta+ 2Kmr_2 \cos2\theta) \ d\theta \Big).
\end{align}
We define 
\begin{align} \label{F def}
    F_{K, m}(r_1, r_2)\coloneqq Kr_1^{2} + Kmr_2^{2} - \log\Big({1 \over 2\pi} \int \exp(2Kr_1 \cos\theta+ 2Kmr_2 \cos2\theta) \ d\theta \Big),
\end{align}
which evidently satisfies $F(0, 0)=0.$ {Critical points of $F_{K,m}$ are exactly solutions of the self-consistency system.} This formulation allows us to restrict our analysis to the finite dimensional space $[-1, 1]^{2},$ significantly simplifying analysis. By considering the transformation $\theta \mapsto \pi - \theta,$ we observe the symmetry
\begin{align} \label{free energy symmetry}
F_{K, m}(r_1, r_2) = F_{K, m}(-r_1, r_2), \quad \forall r_1, r_2 \in [-1, 1].
\end{align}

\subsection{Continuous phase transition for $m \in [0, 1/2]$}\label{continuous phase trans section}
\cite{Onsager} showed there is a {continuous} phase transition at $K_{c}=1$ when $m\leq1/4$. Here, we extend their result to $m \leq 1/2$, which is the content of the main result of this subsection.

\begin{prop} \label{small m continuous theorem}
    If $m \leq {1 / 2}$, then a continuous phase transition occurs at $K_{c}=1$.
\end{prop}

{To prove \cref{small m continuous theorem}, we show: (1) when $K\le 1$, $(r_1,r_2)=(0,0)$ is the unique critical point of $F_{K,m}$ hence the global minimizer; (2) when $K>1$, $\min F_{K,m} < 0$. \cref{phase trans continuity prop} then implies the continuity of the phase transition. Task (2), which is accomplished with \cref{lem:cpnonu}, is through perturbation expansion and not hard. Task (1), which is accomplished with \cref{lem:uniqCP}, is more difficult and requires a couple additional lemmas.}

\begin{lemma}\label{lem:cpnonu}
    If $K>1$ and $m\le {1/2}$, then there exists $(r_1,r_2)\in [-1,1]^2$ such that $F_{K,m}(r_1,r_2)<0$.
\end{lemma}
\begin{proof}
For $\eps \in (0, 1/2)$, observe that
\begin{align}
        F_{K, m}(\eps, 0) = K \eps^{2} - \log \left( {1 \over 2\pi} \int \exp(2K\eps \cos \theta) \ d\theta \right) &= K \eps^{2} - \log \left( 1+ K^{2} \eps^{2} + O(\eps^{3}) \right) \nonumber\\
        &= K(1-K) \eps^{2} + O(\eps^{3}).
\end{align}
        By assumption $K>1$, we conclude that $F_{K, m} (\eps, 0) < F_{K, m}(0, 0)$.
\end{proof}

Having completed Task (2), we turn to Task (1). The first auxiliary lemma  we need shrinks the landscape by showing that the manifold of critical points is one-dimensional.

\begin{lemma} \label{search space shrink}
Suppose that $m \leq {1/2}.$ If $K\leq 1$, then for each $r_1 \in [-1, 1],$ there exists a unique $r_{2}^{*}=r_{2}^{*}(r_1)$ satisfying
\begin{align}
    r_2^{*} =  \frac{ \int \cos 2\theta \ \exp( 2Kr_1 \cos \theta + 2Km r_2^{*} \cos 2\theta) \ d\theta}{ \int \ \exp( 2Kr_1 \cos \theta + 2Km r_2^{*} \cos 2\theta) \ d\theta}.
\end{align}
Furthermore, $r_{2}^{*}(): [-1, 1] \to [-1, 1]$ is a nonnegative, even function.
\end{lemma}
\begin{proof}
Fix $r_1 \in [-1, 1],$ and define the function
\begin{align}
    \psi_{r_1}(r_2)\coloneqq  \frac{ \int \cos 2\theta \ \exp( 2Kr_1 \cos \theta + 2Km r_2 \cos 2\theta) \ d\theta}{ \int \ \exp( 2Kr_1 \cos \theta + 2Km r_2 \cos 2\theta) \ d\theta}.
\end{align}
    
Differentiating $\psi_{r_1},$ we find
\begin{multline}
        \p_{r_2} \psi_{r_1}
        =2Km \left[\int \cos^{2}2\theta  q_{r_1, r_2}(\theta) d\theta - \Big(\int \cos2\theta   q_{r_1, r_2}(\theta) d\theta\Big)^2\right] = 2Km \Var_{\theta \sim q_{r_1, r_2}}[\cos 2\theta]
\end{multline}
where $q_{r_1, r_2}$ is as in \eqref{q def}. Hence,  $0 \leq \p_{r_2} \psi_{r_1} <1$ on $[-1,1]$. Since $\p_{r_2}\psi_{r_1}$ is continuous, it follows that $\|\p_{r_2}\psi_{r_1}\|_{L^\infty} < 1$. The mean-value theorem then implies that $\psi_{r_1}$ is a contraction mapping on $[-1,1]$, hence has a unique fixed point $r_2^*(r_1)$. 
By the symmetry of the free energy \eqref{free energy symmetry}, we conclude that $r_2^{*}$ is an even function in $r_1$.

To see that $r_2^* \ge 0$, note that $r_2^{*}$ is the unique solution to the equation
\begin{align}
    \Psi_{r_1}(r_2)=\int (r_2 - \cos2\theta)  \exp(2K r_1 \cos \theta+ 2Kmr_2 \cos2\theta) \ d\theta =0.
\end{align}
It is evident that that $\Psi_{r_1}(1) >0$ and 
\begin{align}
   \Psi_{r_1}(0) &= - \int \cos2\theta   \exp(2K r_1 \cos \theta) \ d\theta = - 2\pi I_{2}(2Kr_1) \leq 0, 
\end{align}
since $I_{2} \ge 0$. The continuity of $\Psi_{r_1}$ implies that the solution $r_2^{*} \in [0, 1)$.
\end{proof}

Using \cref{search space shrink}, we can show that $(r_1,r_2)=(0,0)$ is the unique critical point when $K\le 1$.

\begin{lemma}\label{lem:uniqCP}
If $m\le {1/2}$ and $K\le 1$, then $(r_1,r_2)=(0,0)$ is the unique critical point of the free energy $F_{K,m}$.
\end{lemma}

To prove \cref{lem:uniqCP}, which will complete the proof of \cref{small m continuous theorem}, we  need the following result from \cite[Theorem 7.3]{Onsager} that shows nonzero critical points are local minimizers.

\begin{lemma} \label{positive hessian in first}
    Suppose $m \leq {1/ 2}$ and $K >0.$ Then at each critical point $(r_1, r_2) \in (0, 1]^{2}$ of $F_{K, m}$, we have $\nabla^{2}F_{K, m}(r_1, r_2) > 0$.
\end{lemma}

\begin{proof}[Proof of \cref{lem:uniqCP}]
To begin, we compute the first and second derivatives of $F_{K,m}$:
\begin{align}
    \partial_{r_j} F_{K, m} = 2K \Big( r_j - \int \cos(j\theta) \ q_{r_1, r_2}(\theta) \ d\theta\Big)
\end{align}
and
\begin{equation} \label{hessian}
\begin{aligned} 
    &\partial_{r_1}^{2} F_{K, m} = 2K  \big(1 - 2K \Var_{\theta \sim q_{r_1, r_2}}[\cos \theta]\big),\\
    &\partial_{r_1} \partial_{r_2}F_{K. m} = -4K^{2}m \cov_{\theta \sim q_{r_1, r_2}}[ \cos\theta, \cos2\theta],\\
    &\partial_{r_2}^{2}F_{K, m} = 2Km \big(1 - 2Km \Var_{\theta \sim q_{r_1, r_2}}[\cos2\theta] \big) .
\end{aligned}
\end{equation}
Evaluating at $(r_1, r_2) = (0, 0)$, we obtain \begin{align}
    \nabla F_{K, m}(0, 0) = \begin{bmatrix}
        0 & 0
    \end{bmatrix}, \quad
    \nabla^{2} F_{K, m} (0, 0) = \begin{bmatrix}
        2K(1-K) & 0\\
        0 & 2Km(1-Km)
    \end{bmatrix}.
\end{align}
We note that the Hessian is strictly positive definite if $K<1.$

If $r_1$ or $r_2$ is zero, then \eqref{sc equation} reduces to the Kuramoto model, implying that the other coordinate must also be zero. In particular, if $r_2=0$, then
\begin{align}
r_1 = \frac{\int \cos\theta \exp\left(2Kr_1 \cos\theta \right) \, d\theta}{ \int \exp\left(2Kr_1 \cos\theta \right) \, d\theta},
\end{align}
which has the unique solution $r_1=0$ by \cref{phase trans in Kura}. Similarly, if $r_1 = 0,$ then
\begin{align}
    r_2 = \frac{\int \cos2\theta \exp\left(2Kmr_2 \cos2\theta \right) \, d\theta}{ \int \exp\left(2Kmr_2 \cos2\theta \right) \, d\theta}= \frac{\int \cos\theta \exp\left(2Kmr_2 \cos\theta \right) \, d\theta}{ \int \exp\left(2Kmr_2 \cos\theta \right) \, d\theta}.
\end{align}
This equation also reduces to the Kuramoto model, yielding the unique solution $r_2 = 0.$ It remains to prove that there is no solution $(r_1, r_2) \in (0, 1]^{2}.$

To this end, we apply the same reasoning as in \cite{Onsager}, using the Poincar\'e-Hopf formula on a big open disk which includes all of critical points. From \Cref{search space shrink,positive hessian in first}, the indices of the non-zero critical points are $0$, which implies the uniqueness of a critical point at $(0, 0)$.
\end{proof}

\subsection{Discontinuous phase transition for $m \in (1/2, 1]$}

In this subsection, we show $m=1/2$ is the maximal value of $m$ for which $K_{c}(m)=1$ holds. Previously, we used a perturbation argument to establish the non-uniqueness of critical point in the supercritical regime $K>1$ when $m \leq 1/2.$ For $m > 1/2,$ we construct a small $(r_1, r_2)$ as a function of $m$ that results in a negative free energy at some $K<1.$ A byproduct of our argument is an upper bound {(almost certainly not sharp)} of $K_{c}(m)$ when $m > 1/2.$ 

\begin{lemma} \label{m near half upper}
If $m>1/2$, then
\begin{align}
        K_{c}(m) <  1 - {36 \over 142897}  \frac{(m^{2} - 1/4)^{2}}{m^{8}}.
\end{align}
\end{lemma}
\begin{proof}
Let $\eps>0.$ Using the lower bound $e^t \geq \sum_{n=0}^4 \frac{t^n}{n!}$, we estimate 
\begin{multline}
    {1 \over 2\pi} \int \exp(2K{\eps}^{1/2} \cos \theta + 2Km \eps \cos2\theta) \, d\theta \geq 1+ \eps K^{2} + \eps^{2} \Big(K^2m^{2}+K^3m+{K^4 \over 4} \Big)\\
    +\eps^{3} \Big({K^{4} m^{2} \over 2} + {K^{5}m \over 3} \Big) +\eps^{4} \Big({K^{4}m^{4} \over 4} + {K^{5}m^{3} \over 2} \Big).
\end{multline}
Using the elementary inequality $\log(1+t) > t - {t^{2} \over 2}$ for $t>0,$ we obtain
\begin{multline}
    \log \Big( {1 \over 2\pi} \int \exp(2K{\eps}^{1/2} \cos\theta + 2Km \eps \cos2\theta) \, dx \Big) \\
    > \eps K^2 + \eps^{2} \Big(K^2 m^{2}+K^3 m-{K^4 \over 4} \Big) - R_{K, m}(\eps),
\end{multline}
where 
\begin{multline}
        R_{K, m} (\eps) \coloneqq \eps^{3} \Big( {1 \over 2}K^{4}m^{2} + {2 \over 3}K^{5}m + {K^{6} \over 4} \Big) + \eps^{4} \Big( {K^{4}m^{4} \over 4} + {K^{5} m^{3} \over 2} + {5 \over 4}K^{6}m^{2} + {7 \over 12}K^{7}m + {1 \over 32}K^{8} \Big)\\
        +\eps^{5} \Big({3 \over 4} K^{6}m^{4} + {4 \over 3}K^{7}m^{3} + {11 \over 24}K^{8}m^{2} + {1 \over 12}K^{9}m \Big)\\
        +\eps^{6} \Big( {K^{6}m^{6} \over 4} + {3 \over 4}K^{7}m^{5} + {11 \over 16} K^{8}m^{4} + {7 \over 24}K^{9}m^{3} + {1 \over 18}K^{10}m^{2} \Big)\\
        +\eps^{7} \Big({K^{8}m^{6} \over 8} + {K^{9}m^{5} \over 3} + {K^{10}m^{4} \over 6} \Big) + \eps^{8} \Big({K^{8}m^{8} \over 32} + {K^{9}m^{7} \over 8} + {K^{10}m^{6} \over 8} \Big).
\end{multline}
    Using $\eps <1, \ K < 1,$ and $1<2m,$ we bound the remainder term $R_{K, m}$ as \begin{align}
        R_{K, m}(\eps) \leq C_{0}  \eps^{3} K^{4} m^{8}, \quad \text{ where} \quad C_{0} \coloneqq \frac{142897}{288}.
    \end{align}
    Recalling the definition  \eqref{F def}  of $F_{K, m}$, we see that
    \begin{align}
    F_{K, m}(\eps^{1/2}, \eps) < \eps  K(1-K) + \eps^{2} \left( Km(1-K^{2}) + K^{2} \Big({K^2 \over 4} - m^{2}\Big) \right) + C_{0} \eps^{3} K^{4} m^{8}.
    \end{align}
    Defining $\delta\coloneqq1-K$ and using that $<1$, it follows that
    \begin{align}
        \nonumber F_{1-\delta, m}(\eps^{1/2}, \eps) &<  \eps \delta + \eps^{2} \left( 2m \delta  - \Big(m^{2} -{1 \over 4} \Big)+2\delta   \Big(m^{2} -{1 \over 4} \Big) \right) + C_{0} \eps^{3}m^{8}\\
        \label{last step} &=  \eps \left[ \delta  \left( 1 + \eps \Big(2m^{2} + 2m - {1 \over 2}\Big) \right) - \left( \Big(m^{2} - {1 \over 4}\Big) \eps - C_{0}m^{8} \eps^{2}\right) \right]. 
    \end{align}
    Now define \begin{align}
        \eps_{\alpha}(m) \coloneqq \alpha  \frac{m^{2} - {1 / 4}}{C_0 m^{8}}, \quad \delta_{\alpha}(m) \coloneqq \frac{ \alpha (1-\alpha) {(m^{2} - {1 / 4})^{2}} }{C_{0}m^{8} + \alpha  {(m^{2} - {1 /4})(2m^{2} + 2m - {1 / 2})} },
    \end{align}
    where $\alpha \in (0, 1]$ is chosen to ensure $\eps_{\alpha}(m), \ \delta_{\alpha}(m) \in (0, 1]$ for all $m > {1 / 2}.$ Since $(m^{2} - {1 / 4})^{2} < C_{0} m^{8}$ for all $m > {1 / 2},$ we can freely choose $\alpha$ in $(0, 1].$ Thus, for $K = 1 - \delta_{\alpha}(m),$ there exists some $\eps_{\alpha}(m) >0$ such that $F_{K, m}(\eps_{\alpha}^{1/2}, \eps_{\alpha}) < 0.$ Hence, setting $\alpha=1/2$,
    \begin{align}
        K_{c} < 1 - \delta_{\alpha}(m)\leq 1 - {1 \over 8C_{0}} \frac{ (m^{2} - {1 \over 4})^{2}}{m^{8}}.
    \end{align}
\end{proof}

Combining \Cref{small m continuous theorem,phase trans continuity prop},  we arrive at the following conclusion:

    \begin{enumerate}  \label{phase trans classification}
        \item If $m \in (0, 1/2],$ the model exhibits a continuous phase transition at $K_{c}=1.$
        \item If $m \in (1/2, 1],$ the model exhibits a discontinuous phase transition at some $K_{c}<1.$ 
    \end{enumerate}

{\subsection{Monotonicity of $K_{c}(m)$}
In this subsection, we show that $K_{c}(\cdot)$ is monotonically decreasing and continuous as a function of $m$. Also, using reasoning similar to \cref{continuous phase trans section}, we show a continuous phase transition when $m$ is sufficiently big by rescaling $\tilde{K} = {Km}$.}   

{
\begin{lemma} \label{monotonicity of K_c}
    $K_{c}(\cdot)$ is non-increasing and continuous.
\end{lemma}
\begin{proof}
    The reader will recall the definition \eqref{original free energy def} of the free energy $\mathcal{F}_{K,m}$.
    Since the entropy is non-negative and each of the cosine integrals is non-negative, it follows that for $K_{1} \leq K_{2}$ and $m_{1} \leq m_{2},$
    \begin{align} \label{Fc monotone}
        \forall q \in \cP(\bT), \quad \cF_{K_{2}, m_{2}}(q) \leq \cF_{K_{1}, m_{1}}(q). 
    \end{align}

    To see that $K_c(\cdot)$ is non-increasing, assume the contrary, $i.e.$ there exist $m_{1}<m_{2}$ such that $K_{c}(m_1) < K_{c}(m_2).$ Since we always have $K_{c}(m) \leq 1,$ it follows that $K_{c}(m_1) < 1.$ Consequently, for $K_{*} \in (K_{c}(m_1), K_{c}(m_2)),$ there is a non-uniform minimizer $q_{*} \in \cP(\bT)$ of $\cF_{K_{*}, m_{1}}$ such that $\cF_{K_{*}, m_{1}}(q_*) < 0.$ Now appealing to the property \eqref{Fc monotone}, we see that \begin{align}
        0 > \cF_{K_{*}, m_{1}}(q_*) > \cF_{K_{*}, m_{2}}(q_*) \geq \min_{\cP(\bT)} \cF_{K_{*}, m_{2}}.
    \end{align}
    Thus, we find $K_{*} \geq K_{c}(m_2),$ which is a contradiction.

    For continuity, observe that since $K_c$ is monotone, it must be continuous except on an at most countable set. Moreover, the function can only have jump discontinuities. Suppose that $m_0$ is a jump discontinuity for $K_{c}(\cdot)$ so that $K_{c}(m_{0}^{+}) < K_{c}(m_{0}^{-}).$

    \begin{enumerate} [(i)]
        \item If $K_{c}(m_0) > K_{c}(m_{0}^{+}),$ then for any $K \in (K_{c}(m_{0}^{+}), K_{c}(m_{0})),$ \eqref{Fc monotone} implies that there exists $\delta_K >0$ and $q_K \in \cP(\bT)$ such that \begin{align}
            \forall m \in [m_0, m_{0} + \delta_K), \quad \cF_{K, m}(q_K) <0. 
        \end{align}
        Taking $m = m_0,$ this implies that $K_{c}(m_0) \leq K,$ which is a contradiction.
        \item On the other hand, if $K_{c}(m_0) < K_{c}(m_{0}^{-}),$ then for any $K \in (K_{c}(m_0), K_{c}(m_{0}^{-})),$ there exists $q_K \in \cP(\bT)$ such that 
        \begin{align}
        \cF_{K, m_0}(q_K) <0.
        \end{align}
        Let $\delta_K>0$ satisfy
        \begin{align}
        {K \delta_K \over 2} \int_{\bT^{2}} \cos(2(\theta - \phi)) \ dq_{K}^{\otimes 2}(\theta, \phi) < {|\cF_{K, m_{0}}(q_K)| \over 2}.
        \end{align}
        Consequently,
        \begin{align}
        \forall m\in [m_0 - \delta_K, m_0], \quad \cF_{K, m}(q_K) \leq {\cF_{K, m_0}(q_K) \over 2}<0.
        \end{align}
        This implies that for any $m \in [m_0 - \delta_K, m_0],$ we have $K_{c}(m) \leq K$ and consequently $K_{c}(m_0^{-}) \leq K,$ which is a contradiction.
    \end{enumerate}
    Combining the two previous cases, we have shown that $K_{c}(m_0^{-}) \leq K_{c}(m_0) \leq K_{c}(m_{0}^{+}),$ which contradicts the assumption that $m_0$ is a jump discontinuity point. Thus, $K_{c}(\cdot)$ is continuous function for every $m \geq 0,$ as claimed.
\end{proof}}

In the next proposition, we show ${(2(1+m))^{-1}}\le K_c(m)\le m^{-1}$. Obviously, this upper bound is only useful when $m\ge 1$. Moreover, when $m\ge 2$, we have equality as well as continuity of the phase transition. {Remark that when $m\le 0$, $K_c(m)=1$, meaning that there is a continuous phase transition. }



\begin{prop} \label{big m phase trans}
If $m>0$, then $ {(2(1+m))^{-1}} \le K_c(m)\le m^{-1}$. If $m\ge 2$, then $K_c(m)=m^{-1}$ and the phase transition is continuous.
\end{prop}
\begin{proof}
We first show the upper bound.  Set $\tilde{K} \coloneqq Km$. Rewriting the free energy in terms of $\tilde{K},$  \begin{align}\label{tilde_F def}
        \tilde{F}_{\tilde{K}, m} (r_1, r_2) = {\tilde{K} \over m} r_{1}^{2} + \tilde{K} r_{2}^{2} - \log \Big( I_{0}({2 \tilde{K}}m^{-1} r_{1}, 2\tilde{K} r_{2}) \Big).
\end{align}
    Similarly to \cref{small m continuous theorem},  let $(r_{1}, r_{2}) = (0, \eps)$ with sufficiently small $\eps>0,$ and $\tilde{K} > 1.$ Then, \begin{align}
        \tilde{F}_{\tilde{K}, m}(0, \eps) = \tilde{K} \eps^{2} - \log \Big( {1 \over 2\pi} \int \exp(2\tilde{K} \eps \cos(\theta)) \ d\theta \Big) =\tilde{K}(1-\tilde{K}) \eps^{2} + O(\eps^{3}).
    \end{align}
    Therefore, $ K_{c} \leq { \tilde{K} / m},$  and letting $K \to 1+$ we obtain the desired result. 

\medskip
We now show the lower bound.  We will show that for $K \leq (2(1+m))^{-1},$ the free energy $F_{K, m}$ is strictly convex in $[-1, 1]^{2}.$

From \eqref{hessian}, $\nabla^{2} F_{K, m}(r_{1}, r_{2}) >0$ is equivalent to
\begin{align}
        &1 - 2K  \Var[\cos \theta] >0, \label{eq:hesspd1}\\
         &(1 - 2K  \Var_{\theta }[\cos \theta])(1 - 2Km  \Var\cos 2\theta]) - 4K^{2}m  (\cov[\cos\theta, \cos 2\theta])^{2}>0. \label{eq:hesspd2}
\end{align}
where $\theta\sim q_{r_1,r_2}$.   The first inequality holds since $K<{1/2}$. For the second inequality, let us lighten the notation by setting
\begin{align}
        v_{11} \coloneqq \Var[\cos\theta], \quad v_{22} \coloneqq \Var[\cos2\theta], \quad v_{12} \coloneqq \cov[\cos\theta, \cos 2\theta].
\end{align}
We this notation, we rewrite \eqref{eq:hesspd2} as
\begin{align}\label{intermediate big m}
        4m \big(v_{11} v_{22} - v_{12}^2\big) K^{2}
        - 2(v_{11} +m v_{22}) K +1 >0.
\end{align}
A sufficient condition for \eqref{intermediate big m} is \begin{align}
            K < \left[v_{11} + mv_{22} + \sqrt{(v_{11} + m v_{22})^{2} - 4m(v_{11} v_{22} - v_{12}^{2})}\right]^{-1}.
        \end{align}
Thus, it suffices to show the expression inside the brackets is $<2(1+m)$.
By Cauchy-Schwarz, $v_{11} v_{22} - v_{12}^{2}>0$, which gives \begin{align}
             v_{11} + mv_{22} + \sqrt{(v_{11} + m v_{22})^{2} - 4m(v_{11} v_{22} - v_{12}^{2})} < 2(v_{11}+mv_{22}) \leq 2(1+m).
\end{align}

\medskip
Lastly, we show the equality and continuity of the phase transition when $m\ge 2$. By \cref{phase trans continuity prop}, it suffices to show that if $\tilde{K} = 1,$ then $(0, 0)$ is the unique critical point of $\tilde{F}_{\tilde{K}, m}.$

Let $\tilde{K}=1$.
Similarly to \cref{search space shrink}, for fixed $r_2 \in [-1, 1],$ define the map
\begin{align}
    \varphi_{r_2}(r_1) \coloneqq \frac{\int \cos \theta \exp \big( {2 \over m} r_{1} \cos \theta + 2 r_{2} \cos 2\theta \big) \ d\theta}{\int \exp \big( {2 \over m} r_{1} \cos \theta + 2 r_{2} \cos 2\theta \big) \ d\theta }
    \end{align}
and observe that
\begin{align}
        \varphi_{r_2}'(r_1) = {2 \over m} \Var_{\theta \sim \mu}[\cos \theta], \quad \mu(\theta) \coloneqq \frac{ \exp \big( {2 \over m} r_{1} \cos \theta + 2 r_{2} \cos 2\theta \big)} {\int \exp \big( {2 \over m} r_{1} \cos \theta + 2 r_{2} \cos 2\theta \big) \ d\theta }.
    \end{align}
    By our assumption $m \geq 2$ and the continuity of $\varphi_{r_2}'$, we have $\|\varphi_{r_2}'\|_{L^\infty}<1$, which implies that $\varphi_{r_2}$ is a contraction and has a unique fixed point $r_{1}^{*} = r_{1}^{*}(r_2)$. Since $r_1=0$ is also a fixed point of $\varphi_{r_2}'$, the uniqueness of $r_1^*$ implies that $r_1^*(r_2) = 0$ for every $r_2\in [-1,1]$.

    Finally, let $(r_1, r_2)$ be a critical point of $\tilde{F}_{, m}.$ The self-consistency equation \eqref{sc equation} gives $r_1 = \varphi_{r_2}(r_1),$ so $r_1 = 0.$  Moreover, \begin{align}
       2  \Big[r_{2} - \frac{I_{1}(2r_2)}{I_{0}(2r_2)} \Big] = \partial_{r_2} \tilde{F}_{1, m} (0, r_2) = 0.
    \end{align}
    Using \eqref{Kuramoto sc}, we have $r_2 = 0.$ As a result, $(0, 0)$ is the unique critical point. 
\end{proof}


{
\subsection{Clarification of $m_*$}\label{ssec:PTm*}
In this last subsection, we further investigate $K_{c}(m)$ for the remaining values $m \in (1, 2),$ identifying that there exists a critical $m_{*}$ that distinguishes whether $K_{c} = m^{-1}$ or not. We study the minimizer of the free energy $F_{K, m}(r_1, r_2)$ on the domain $[-1, 1]^{2}.$ For all $K, m>0,$ no minimizer can lie on the boundary $\{r_{1} = \pm 1\} \cup \{r_2 = \pm 1\},$ since $\cos(j\theta) = \pm 1$ $q_{r_1, r_2}$-a.s. cannot hold. 

The main result of this subsection is the following proposition. 

\begin{prop}\label{prop:m*}
    There exists $m_*  \in (1,2]$ such that the following hold.
\begin{enumerate}[(1)]
    \item\label{item:m*1} If $1<m<m_*$, then $K_c<m^{-1}$ and the phase transition is discontinuous.
    \item\label{item:m*2} If $m=m_*$, then $K_c=m^{-1}$ and the phase transition is discontinuous. 
    \item\label{item:m*3} If $m>m_*$, then $K_c=m^{-1}$ and the phase transition is continuous.
\end{enumerate}
\end{prop}

To prove \cref{prop:m*}, it will be convenient to set
\begin{align}
        G_{m}(r_1, r_2) \coloneqq F_{m^{-1}, m} (r_1, r_2)= {r_1^{2} \over m} + r_{2}^{2} - \log \Big( {1 \over 2\pi} \int_{0}^{2\pi} \exp( {2 \over m} r_1 \cos \theta + 2r_{2} \cos 2\theta) d\theta \Big).
\end{align}
Let $s(m) \coloneqq \min \{G_{m}(r_1, r_2): (r_1, r_2) \in [-1, 1]^{2}\}.$  Combining \Cref{big m phase trans,m near half upper}, we have
\begin{align}
    s(m) \quad \begin{cases} \le 0, & {\forall m} \\ < 0, & {m=1}   \\ =0, & \forall m\ge 2.\end{cases} 
\end{align}
Define \begin{align}
        m_{*} \coloneqq \inf\{ m \leq 2: s(m) = 0 \}.
    \end{align}
From the continuity of $s(m),$ we observe $m_{*} \in (1, 2]$ and \begin{align}
        s(m) \quad  \begin{cases}
            <0 & \text{ if } m<m_{*},\\
            =0 & \text{ if } m\geq m_{*}.
        \end{cases}
    \end{align}
    
We will make use of the following lemma, which describes the behavior of $G_{m}$ near the origin. 

\begin{lemma} \label{G near origin}
For each $\alpha \in (0, 1)$, there exists $\eps = \eps(\alpha) >0$ such that
\begin{align}
    \forall m \in [1 + \alpha, 2], \qquad  G_{m}(r_1, r_2) >0, \quad \text{for } (r_1, r_2) \in B_{\eps}(0, 0) \setminus (0, 0), 
\end{align}
\end{lemma}

Assuming \cref{G near origin}, let us prove \cref{prop:m*}.
\begin{proof}[Proof of \cref{prop:m*}]
If $m\in (1,m_*)$, then $F_{K,m}$ has a nonuniform global minimizer when $K=m^{-1}$. This implies assertion \ref{item:m*1}.

\medskip
For assertion \ref{item:m*3}, it suffices to show $G_m$ has the unique global minimizer $(0, 0).$ Assume the contrary, i.e. there exists global minimizer $(r_1, r_2) \neq (0, 0)$. Consider the following cases.
    \begin{enumerate}[(i)]
        \item If $r_1 =0,$ then
        \begin{align}
            G_{m}(0, r_2) = r_{2}^{2} - \log( I_{0}(2r_{2}))
        \end{align}
        cannot be equal to $0$ unless $r_2 = 0,$ which is a contradiction.
        \item If $r_1 \neq 0$, then
        \begin{align}
            \p_m G_{m}(r_{1}, r_{2}) &= -\frac{r_{1}^{2}}{m} + \frac{2 r_{1}}{m^{2}}  \frac{I_{1}(2r_{1}/m, 2r_{2})}{I_{0}(2r_{1}/m, 2r_{2})} ={r_{1}^{2} \over m^{2}} >0.
        \end{align}
Hence, there exists $m' \in (m_*, m)$ such that $G_{m'}(r_1, r_2) <0.$ This contradicts $s(m') =0.$
    \end{enumerate}
The two cases form dichotomy, completing the argument.

\medskip
Finally, consider assertion \ref{item:m*2}. \cref{monotonicity of K_c} implies that $K_{c}(m_*) = m_{*}^{-1}.$ Since $s(m_*) = 0$, the uniform distribution is a global minimizer of $G_{m_*}.$ We prove that the phase transition is discontinuous by showing the existence of non-unique global minimizers of $G_{m_*}$.

Assume the contrary, i.e. $G_{m_*}$ has the unique minimizer $(0, 0).$ Since $m_{*} \in (1, 2],$ we may apply \cref{G near origin} to obtain a neighborhood $U$ of $(0, 0)$ such that $G_{m} >0$ in $U \setminus (0, 0)$ for all $m \in [(1+m_*)/2, 2].$ Since $[-1, 1]^{2} \setminus U$ is compact and $G_{m_*}>0$ on this set, we can find $\delta>0$ such that for $ m \in (m_{*} - \delta, m_{*}],$ \begin{align}
        G_{m} >0 \quad \text{in } [-1, 1]^{2} \setminus (0, 0)
    \end{align}
    which contradicts $s(m)<0$ for all $m \in (1, m_{*}).$ 
\end{proof}

We pay our debt to the reader by proving  \cref{G near origin}.

\begin{proof}[Proof of \cref{G near origin}]
       
      For $\mathbf{r} = (r_1, r_2)$ near $(0, 0),$ we decompose $G_{m}(\fr) = P_{m}(\fr) + R_{m}(\fr)$, where
      \begin{align}
           P_{m}(\fr) \coloneqq \frac{m-1}{m^{2}} r_{1}^{2} - {1 \over m^{2}}r_{1}^{2}r_{2} + {1 \over 4m^{4}} r_{1}^{4} + {1 \over 4} r_{2}^{4}, \quad R_{m}(\fr) \coloneqq O( | \fr |^{5}).
       \end{align}
       First, we claim that there exists $\delta_{1}\coloneqq (m-1)/2$ such that for all $\delta< \delta_1,$ \begin{align}
           P_{m}(\fr) \geq {m-1 \over 4m^{2}} (r_1^{2} + r_{2}^{4}), \quad \forall | \fr| < \delta.
       \end{align}

       We fix $|\fr|<\delta$ so that $|r_1|, |r_2| < \delta < \delta_1.$ \begin{enumerate}[(i)]
           \item If $|r_2|> |r_1|^{1/2},$ then \begin{align}
               |r_1^{2} r_{2}| \leq |r_{2}|^{5} \leq \delta_{1} r_{2}^{4}.
           \end{align}
           Hence, \begin{align}
               P_{m}(\fr) \geq \frac{m-1}{m^{2}} r_{1}^{2} + \left({1 \over 4}  - {\delta_1 \over m^{2}}\right) r_{2}^{4} \geq {m-1 \over 4m^{2}} (r_1^{2} + r_{2}^{4}),
           \end{align}
           by choosing $\delta_1 = (m-1) /2.$
           \item On the other hand, if $|r_2| \leq |r_1|^{1/2},$ then \begin{align}
               |r_1^{2} r_{2}| \leq \delta_{1} r_{1}^{2},
           \end{align}
           and we obtain \begin{align}
               P_{m}(\fr) \geq \frac{m-1 - \delta_{1}}{m^{2}} r_{1}^{2} + {1 \over 4}  r_{2}^{4} \geq {m-1 \over 4m^{2}} (r_1^{2} + r_{2}^{4}).
           \end{align}
       \end{enumerate}

       Moreover, we find $\eps = \eps(\alpha)>0$ and a constant $C>0$ such that
       \begin{align}
       R_{m}(\fr) \leq C \eps (r_{1}^{2} + r_{2}^{4}), \quad \forall |\fr|< \epsilon,
       \end{align}
       since $|r_1|^{5} \leq \eps^{3} r_{1}^{2} $ and $|r_{2}|^{5} \leq \eps r_{2}^{4}.$ Combining the two estimates, we have
       \begin{align}
       G_{m}(\fr) \geq \left( {m-1 \over 4m^{2}} - C \eps \right) (r_{1}^{2} + r_{2}^{4}),
       \end{align}
       which finishes the proof.
       
   \end{proof}

}

\section{Linear stability analysis} \label{section 4}
Having verified the critical interaction strength $K_{c}$ in terms of $m$, we now investigate the stability of the stationary solution $q$. The goal of this section is the proof of \cref{stability thrm}.

In the following, we consider the supercritical regime $K>1$.
To access  \cref{exact multiplicities}, we will need to restrict to $m\le 1/4$. For technical reasons, we will further restrict to the smaller range on $m\le m_0$. We continue to let $(r_1(K, m), r_2(K, m))$ be the unique solution of the self-consistency equation \eqref{sc equation} in $(0, 1)^{2}$,
and let $q\coloneqq q_{r_1, r_2}$ be the associated probability measure. Following the notation of \cref{Section 2},
we write $\bar{r}_{1}\coloneqq \bar{r}_{1}(K) = r_1(K, 0)$ and $\bar{q} = q_{\bar{r}_1,0}$. The reader will also recall the function space notation from \cref{ssec:intro_contib}.

\subsection{Outline of proof}\label{ssec:stab_out}
Before diving into the proof, let us outline the overall argument. Our goal is to establish a lower bound for the Dirichlet form
\begin{align}\label{eq:DFdef}
    \forall u\in \mathcal{Q}(L_q), \qquad \mathbf{D}(u) \coloneqq - \inner{u}{L_{q}u}_{H^{-1}_{1/q}} ={1 \over 2} {\inner{u}{u}_{L^{2}_{1/q}} - \inner{u} {q{J}*u}_{L_{1/q}^2}},
\end{align}
where $\mathcal{Q}(L_q)$ denotes the form domain of $L_q$ and
\begin{align}\label{eq:Jdef}
    J(\theta) \coloneqq K W(\theta) = K(\cos \theta + m \cos 2\theta).
\end{align}
{Remark that the right-hand side of \eqref{eq:DFdef} is precisely the second variation of the free energy $\mathcal{F}$.}

The first, and easier, task is to clarify the functional analytic setting and spectral properties of the operator $L_q$ and the associated Dirichlet form $\bf{D}$. This is the content of \cref{functional section} and follows from standard results, treating the nonlinear part as a perturbation of the linear diffusion. The much more difficult task is estimating the value of the spectral gap, i.e. the lower bound for $\bf{D}$ after modding out the kernel of $L_q$, which we break into several steps.

It will be convenient to introduce the following notation for the covariance of Fourier modes: for $j,l\in\bZ$,
\begin{align}\label{eq:DbarDdef}
    D_{jl}(K,m) \coloneqq \int \cos(j\theta)\cos(l\theta)dq(\theta), \qquad \bar{D}_{jl}(K) \coloneqq \int \cos(j\theta)\cos(l\theta)d\bar{q}(\theta).
\end{align}
In the sequel, we will omit the explicit dependence on $K,m$. 

Rewriting the self-consistency equations \eqref{sc equation} as
\begin{align}
    r_{1} = {I_{1}(2Kr_1, 2Kmr_2) \over I_{0}(2Kr_1, 2Kmr_2)}, \quad r_{2} = {I_{2}(2Kr_1, 2Kmr_2) \over I_{0}(2Kr_1, 2Kmr_2)},
\end{align}
one sees that for $j=1,2$,
\begin{align}
    r_j = \int \cos(j\theta)\ dq(\theta), \qquad \bar{r}_j = \int \cos(j\theta) \ d\bar{q}(\theta).
\end{align}
It is an easy integration by parts to check that (see \cite[(2.21)]{bertini})
\begin{align}\label{Kuramoto variables}
        \bar{r}_{2} = 1 - {1 \over K}. 
\end{align}
Using integration by parts in \eqref{sc equation}, we may obtain explicit expressions relating $D_{11}, \ D_{12}, \ D_{22}$ to $r_{1}, r_{2}$ in the form of \cref{D expression lemma}. In particular, we establish an $O(K^{-1})$ bound for the variance of the Fourier coefficients of the state $q$. This bound will be crucial for estimating $r_{1},r_{2}$ in terms of $m$. The proof of \cref{D expression lemma} is a routine calculation, which, for completeness of exposition, we give in {\cref{D expression lemma proof section}}.

\begin{lemma} \label{D expression lemma}
    For $K>1$, we have
    \begin{equation} \label{D formulas}
    \begin{aligned}
        &D_{11} = \frac{1+r_2}{2}, \qquad \bar{D}_{11} = 1-\frac{1}{2K}\\
        &D_{12} = r_{1} \left(1 - \frac{r_{2} - \bar{r}_{2}}{4mr_{2}} \right), \qquad \bar{D}_{12} = \bar{r}_{1} - {\bar{r}_{2} \over K \bar{r}_{1}}\\  
        &D_{22} = 1 -{1 \over 2Kmr_{2}} \left(r_{2} - Kr_{1}^{2}  \frac{r_{2} - \bar{r}_{2}}{4mr_{2}} \right), \qquad \bar{D}_{22} = 1 - {2 \over K} + {3 \bar{r}_{2} \over K^{2} \bar{r}_{1}^{2}}.
    \end{aligned}
    \end{equation}
    Furthermore,
    \begin{align}
        \label{IBP first} &r_{1} = 2Kr_{1} (1-D_{11}) + 4Kmr_{2} (r_{1} - D_{12}), \\
        \label{IBP second} &r_{2} = Kr_{1} (r_{1} - D_{12}) + 2Kmr_{2}(1-D_{22}),\\
       \label{IBP third} &1 - 2K(1-D_{11})-2Km(1-D_{22}) + 4K^{2}m \big( (1-D_{11})(1-D_{22}) - (r_{1} - D_{12})^{2} \big) = 0.
    \end{align}
\end{lemma}

Recalling the expression \eqref{eq:DFdef} of the Dirichlet form in terms of the $L_{1/q}^2$ inner product, the first important step to lower bounding the spectral gap is the following orthogonal decomposition of $L_{1/q}^2$. One should compare \cref{decomposition} with the simpler decomposition \cite[Lemma 2.1]{bertini} in that suffices for the Kuramoto model (i.e. when $m=0$). The proof of \cref{decomposition} is given in \cref{ssec:stab_od}.

\begin{prop} \label{decomposition}
    Suppose $K>1$. We have the orthogonal decomposition
    \begin{align}
        L^{2}_{1/q} = V_{0} \oplus V_{1/2} \oplus V_{\mu} \oplus V_{\lambda_{+}} \oplus V_{\lambda_{-}},
    \end{align}
    where
    \begin{align}
        \label{mu def} &\mu \coloneqq {K \over 2} \Big( (1-D_{11}) + m (1-D_{22}) - \sqrt{ [(1-D_{11}) - m (1-D_{22})]^{2} + 4m(r_1 - D_{12})^{2} } \Big), \\
       \label{lambda def} &\lambda_{\pm} \coloneqq {K \over 2} \Big( D_{11} + m D_{22} \pm \sqrt{(D_{11} - m D_{22})^{2} + 4m D_{12}^{2}} \Big),\\
       \label{eq:V0def} &V_{0} \coloneqq \left\{ u_{0} =  a_0 + \sum_{j=3}^{\infty} (a_j \cos j \theta + b_j \sin j \theta) : a_0, a_j, b_j \in\bR\right\},\\
       \label{eq:V12def} &V_{1/2} \coloneqq \left\{ u_{1/2}=c_{1/2}  q (r_1 \sin \theta + 2mr_2 \sin 2\theta)   : c_{1/2} \in \mathbb{R} \right\}\\    \label{eq:Vmudef} &V_{\mu} \coloneqq \left\{ u_{\mu} =c_{\mu}  q  (-2r_2 \sin \theta + r_1 \sin 2\theta) : c_{\mu} \in \mathbb{R} \right\}, \\   \label{eq:Vlambdef}  &V_{\lambda_{\pm}} \coloneqq \left\{ u_{\lambda_{\pm}} = c_{\lambda_{\pm}}  q(\eta_{1}^{\pm} \cos \theta + \eta_{2}^{\pm} \cos 2\theta)  : c_{\lambda_{\pm}} \in \mathbb{R} \right\}.
    \end{align} 
    and
    \begin{equation} \label{eta def}
    \begin{aligned}
        &\eta_{1}^{\pm}\coloneqq D_{12}, \\
        &\eta_{2}^{\pm} \coloneqq {1 \over 2} \Big( -(D_{11} - m D_{22}) \pm \sqrt{(D_{11} - m D_{22})^{2} + 4m D_{12}^{2}} \Big).
    \end{aligned}
    \end{equation}
    Moreover, each subspace is an eigenspace of operator $u \mapsto q ({J} * u)$ satisfying
    \begin{align}
        q ({J} * u) = \lambda  u \quad \text{ for } u \in V_{\lambda}, \quad \lambda \in \left\{0, {1 \over 2}, \mu, \lambda_{+}, \lambda_{-} \right\}.
    \end{align}
\end{prop}

\begin{remark}
    The reader may check that $V_{1/2} = \mathrm{span}\{q'\} \subset \ker L_q$. A consequence of \cref{Dirichlet lower} below is that, in fact, $\mathrm{span}\{q'\}  = \ker L_q$.
\end{remark}

When $m=0$ (i.e. Kuramoto), the eigenvalues are explicitly given by $\lambda_{+} = K-1/2$ and $\lambda_{-} = \mu = 0$. Moreover, the quantities $\eta_j^{\pm}$ are also explicit. In the general case $m>0$, we  no longer have explicit values. Nevertheless, if $m$ is sufficiently small, then at the cost of more involved computations,  we can show good perturbative estimates around the $m=0$ case. For instance, see \cref{lem:evalests}.

Applying \cref{decomposition} to \eqref{eq:DFdef}, we see that
\begin{multline}
     \mathbf{D}(u)   = \frac{1}{2} \|u_0\| _{L^{2}_{1/q}}^2
    + \Big(\frac{1}{2} - \mu\Big) \|u_\mu\|_{L^{2}_{1/q}}^2
    + \Big(\frac{1}{2} - \lambda_{+}\Big) \|u_{\lambda_{+}}\|_{L^{2}_{1/q}}^2
   + \Big(\frac{1}{2} - \lambda_{-}\Big) \|u_{\lambda_{-}}\|_{L^{2}_{1/q}}^2.
\end{multline}
Remark that the $u_{1/2}$ contribution cancels because it is an eigenvector of $qJ\ast(\cdot)$ with eigenvalue $1/2$. Writing $u_0 = a_0(1+ v_0)$, where $\int v_0 = 0$, the next proposition allows us to compute explicitly the minimal value of $\langle{u_0,u_0}\rangle$, which is achieved in $V_0$. The proof is presented in \cref{ssec:opt_in_V0}. Existence is by standard convex optimization, while the explicit minimal value is by direct computation.

\begin{prop} \label{opt in V0}
    Suppose $K>1$ and $m \in [0, 1/4]$. There exists a unique $\hat{v}_0 \in V_{0}$ with $\int \hat{v}_0 =0$ such that
    \begin{align}\label{eq:opt in V0}
        \min_{v_0 \in V_{0}: \int v_0=0} \|1+v_0\|_{L_{1/q}^2}^2 = \|1+\hat{v}_0\|_{L_{1/q}^2}^2
        = \frac{(2\pi)^{2}(D_{11}D_{22}-D_{12}^{2})}{(D_{11} - r_{1}^{2})(D_{22} - r_{2}^{2}) -(D_{12} - r_{1}r_{2})^{2}}.
    \end{align}
\end{prop}

If $\int u = 0$, then inserting the explicit form of $u_0,u_{1/2},u_{\lambda_\pm},u_{\mu}$,
\begin{align}
        2\pi a_{0} + (c_{\lambda_{+}} \eta_{1}^{+} + c_{\lambda_{-}} \eta^{-}_{1}) \int  \cos \theta \ dq +(c_{\lambda_{+}} \eta_{2}^{+} + c_{\lambda_{-}} \eta^{-}_{2}) \int \cos 2\theta \ dq= 0.
\end{align}
Yielding $a_0$ explicitly in terms of $r_j,\eta_j^{\pm},c_{\lambda_\pm}$,
 \begin{align}\label{eq:a0}
        a_{0} = - {1 \over 2\pi}\big((r_{1} \eta_{1}^{+} + r_{2} \eta_{2}^{+}) c_{\lambda_{+}} + (r_{1} \eta^{-}_{1} + r_{2} \eta^{-}_{2}) c_{\lambda_{-}} \big).
    \end{align}
It follows now that
\begin{multline}\label{eq:Drewrite}
     \mathbf{D}(u)  = \frac{a_0^2}{2}\|1+\hat{v}_0\|_{L_{1/q}^2}^2 + \Big(\frac12-\mu\Big)\|u_\mu\|_{L_{1/q}^2}^2 + \Big(\frac12-\lambda_+\Big)\|u_{\lambda_+}\|_{L_{1/q}^2}^2 + \Big(\frac12-\lambda_-\Big)\|u_{\lambda_-}\|_{L_{1/q}^2}^2 \\
     +\underbrace{\frac{a_0^2}{2}\Big(\|1+v_0\|_{L_{1/q}^2}^2 - \|1+\hat{v}_0\|_{L_{1/q}^2}^2\Big)}_{\ge 0}.
\end{multline}
We may use the expression \eqref{eq:a0} for $a_0$, the expression \eqref{eq:opt in V0} for $\|1+\hat{v}_0\|_{L_{1/q}^2}^2$, and the explicit forms of \eqref{eq:Vmudef}, \eqref{eq:Vlambdef} of $V_\mu, V_{\lambda_{\pm}}$ to  rewrite the first line of the preceding right-hand side as
\begin{align}\label{eq:Dmform}
    \begin{bmatrix}
            c_{\lambda_{+}} & c_{\lambda_{-}} 
        \end{bmatrix}
        \fC
        \begin{bmatrix}
            c_{\lambda_{+}}\\
            c_{\lambda_{-}} 
        \end{bmatrix}
        + \Big({1 \over 2} - \mu \Big) \alpha_{\mu} c_{\mu}^{2},
\end{align}
where the matrix $\fC = \begin{bmatrix}
        \fC^{++} & \fC^{+-} \\
        \fC^{+-} & \fC^{--}
    \end{bmatrix}$
is defined by
\begin{align}
        \label{eq:C++def} &\mathbf{C}^{++} \coloneqq   \frac{D_{11}D_{22} - D_{12}^{2}}{2\big((D_{11} - r_{1}^{2})(D_{22} - r_{2}^{2}) -(D_{12} - r_{1}r_{2})^{2}\big)} (r_{1}\eta_{1}^{+} + r_{2} \eta_{2}^{+})^{2} + \Big({1 \over 2} - \lambda_{+} \Big) \alpha_{\lambda_{+}},\\
        \label{eq:C+-def} &\mathbf{C}^{+-} \coloneqq  \frac{D_{11}D_{22} - D_{12}^{2}}{2\big((D_{11} - r_{1}^{2})(D_{22} - r_{2}^{2}) -(D_{12} - r_{1}r_{2})^{2}\big)}  (r_{1}\eta_{1}^{+} + r_{2} \eta_{2}^{+}) (r_{1}\eta^{-}_{1} + r_{2} \eta^{-}_{2}),\\
        \label{eq:C--def} &\mathbf{C}^{--} \coloneqq   \frac{D_{11}D_{22} - D_{12}^{2}}{2\big((D_{11} - r_{1}^{2})(D_{22} - r_{2}^{2}) -(D_{12} - r_{1}r_{2})^{2}\big)}  (r_{1}\eta^{-}_{1} + r_{2} \eta^{-}_{2})^{2} + \Big({1 \over 2} - \lambda_{-} \Big) \alpha_{\lambda_{-}}
    \end{align}
and
\begin{align}
        \label{eq:al12def} &\alpha_{1/2}\coloneqq \| q(\theta) (r_1 \sin \theta + 2mr_2 \sin 2\theta)\|^{2}_{L^{2}_{1/q}},\\
        \label{eq:almudef} &\alpha_{\mu}\coloneqq \| q(\theta) (-2r_2 \sin \theta + r_1 \sin 2\theta)\|^{2}_{L^{2}_{1/q}}, \\
        \label{eq:alladef} &\alpha_{\lambda_{\pm}}\coloneqq \| q(\theta)  (\eta_{1}^{\pm} \cos \theta + \eta_{2}^{\pm} \cos 2\theta) \|^{2}_{L^{2}_{1/q}}.
\end{align}
Further expressions for $\alpha_{\lambda_\pm}$ are given in \eqref{eq:alla+}, \eqref{alpha_2 expression}.

It turns out that the matrix $\fC$ is strictly positive definite and therefore has least eigenvalue $\gamma>0$. This nonobvious property is primarily a consequence of \cref{lem:maindet}, the proof of which requires several auxiliary lemmas and constitutes the bulk of the technical difficulty of this paper. In order not to obfuscate the main argument, we defer these technical results to \cref{sec:Tech}; and for the purposes of this section, the reader may take these results as a black box.  Reiterating a point made above, the absence of explicit values for quantities such as $\mu,\lambda_\pm, \eta_j^{\pm}$  substantially increases the involvedness of the computations compared to the $m=0$ case.

    
    
    Using the positive definitess of $\fC$ and that $\mu<1/2$ on \eqref{eq:Dmform}, then reversing  our steps in going from \eqref{eq:Drewrite} to \eqref{eq:Dmform},  we arrive at the following proposition. The complete details are given in \cref{ssec:stab_LB}.

\begin{prop} \label{Dirichlet lower}
    Suppose $K>1$ and $m \in [0,m_{0}]$. Then there exists $C_{1} \in (0, 1/2]$ such that
    \begin{align}\label{eq:Dirichlet lower}
        \forall u\in L_0^2, \qquad  \mathbf{D}(u) \geq C_{1} \|u-u_{1/2}\|^{2}_{L^{2}_{1/q}}, \qquad u_{1/2} = \frac{\inner{u}{q'}_{L^{2}_{1/q}}}{\inner{q'}{q'}_{L^{2}_{1/q}}}  q'.
    \end{align}
     In particular, this implies $\mathbf{D}(u) \geq 0$ and $\ker L_q = \mathrm{span} \{q'\}$. 
\end{prop}

With \cref{Dirichlet lower} in hand, the conclusion of the proof of \cref{stability thrm} is a fairly straightforward consequence of the embedding $\|\cdot\|_{H_{1/q}^{-1}}\le C\|\cdot\|_{L_{1/q}^2}$. The details are given in \cref{ssec:stab_conc}.

\subsection{Self-adjointness and spectral properties of $L_{q}$}\label{functional section}
Having sketched the proof of \cref{stability thrm} in the previous subsection, we now turn to its proof. In this subsection, we clarify the functional analytic setting and spectral properties of the operator $L_q$ introduced in \eqref{linear MV operator}. This will establish the first assertion of \cref{stability thrm}. 

Starting from the right-hand side of the Dirichlet form \eqref{eq:DFdef}, which we decompose as
\begin{align}
    \mathbf{D}(u) = {1 \over 2} {\inner{u}{u}_{L^{2}_{1/q}} - \inner{u} {q{J}*u}_{L_{1/q}^2}} \eqqcolon \mathbf{D}_1(u) + \mathbf{D}_2(u).
\end{align}
 Clearly, this expression makes sense for any $u\in L^2$.

Integrating by parts, we may rewrite
\begin{align}
    \mathbf{D}_1(u) = -\frac12\langle{u,\p_\theta(q\p_\theta(q^{-1}u))}\rangle_{H_{1/q}^{-1}} \eqqcolon \langle{u,L_{q,1}u}\rangle_{H_{1/q}^{-1}}.
\end{align}
Since $q$ is smooth and strictly positive, it is evident that $L_{q,1}$ is a positive, self-adjoint densely defined operator on $H_{1/q}^{-1}$ with $\mathrm{Dom}(L_{q,1}) = H_q^1$ and form domain $\mathcal{Q}(L_{q,1}) = L_{1/q}^2$. Moreover, $L_{q,1}$ has compact resolvent, therefore by \cite[Theorem 6.29]{kato2013perturbation}, it has only pure point spectrum $0<\tilde{\lambda}_{n} \le \tilde{\lambda}_{n+1}$ with associated eigenvectors $\tilde{e}_n$ forming a complete orthonormal basis of $H_{1/q}^{-1}$. One can also show that the $\tilde{e}_n$ are smooth. 

Similarly, we may rewrite 
\begin{align}
    \mathbf{D}_2(u) = \langle{u,\p_\theta(q\p_\theta J\ast u)}\rangle_{H_{1/q}^{-1}} \eqqcolon \langle{u,L_{q,2}u}\rangle_{H_{1/q}^{-1}}.
\end{align}
It is easy to see that $L_{q,2}$ is symmetric and
\begin{align}
    \|L_{q,2}u\|_{H_{1/q}^{-1}} = \|\p_\theta J\ast u\|_{L_q^2} \le \|q\|_{L^\infty} \|\p_\theta J \ast u\|_{L^2} &\le \|q\|_{L^\infty}^{1/2} \|\p_\theta^2 J\|_{L^1} \|\p_\theta^{-1}u\|_{L^2} \nonumber\\
    &\le \|q\|_{L^\infty} \|\p_\theta^2 J\|_{L^1}\|u\|_{H_{1/q}^{-1}}.
\end{align}
By the Kato-Rellich theorem \cite[Theorem X.12]{reed_methods_1975}, $L_q = L_{q,1}+L_{q,2}$ is a densely defined, self-adjoint operator on $H_{1/q}^{-1}$ with $\mathrm{Dom}(L_q) = H_{q}^1$. Moreover, by the KLMN theorem \cite[Theorem X.17]{reed_methods_1975}, the form domain $\mathcal{Q}(L_{q}) = L_{1/q}^2$. Since $L_{q,2}$ is a compact operator on $H_{1/q}^{-1}$, $L_{q}$ also has compact resolvent and therefore has only pure point spectrum with (smooth) eigenvectors forming a complete orthonormal basis of $H_{1/q}^{-1}$.

\subsection{Proof of \cref{D expression lemma}} \label{D expression lemma proof section}
In this subsection, we prove \cref{D expression lemma}.

The formula $D_{11} = (1+r_2)/2$ follows immediately from the identity $\cos^{2}\theta = (1+ \cos2\theta)/2.$ For the $D_{12}$ identity, we consider $r_1$ identity from \eqref{sc equation} and integrate by parts to obtain
    \begin{align}
        &r_{1} \int \exp(2Kr_{1} \cos\theta + 2Kmr_{2} \cos2\theta) \ d\theta \nonumber\\
        &= \int \cos \theta  \exp(2Kr_{1} \cos\theta + 2Kmr_{2} \cos2\theta) \ d\theta \nonumber\\
        &= - \int \sin \theta (-2Kr_{1} \sin \theta -4Km r_{2} \sin2\theta)  \exp(2Kr_{1} \cos\theta + 2Kmr_{2} \cos2\theta) \ d\theta.
    \end{align}
    Using $\sin ^{2} \theta = (1- \cos2\theta)/2$ and $\sin\theta \sin2\theta = \cos \theta - \cos \theta \cos 2\theta$, it follows that 
    \begin{align}\label{eq:r1D12}
        r_{1}=Kr_{1} (1-r_{2}) + 4Kmr_{2} (r_{1} - D_{12}),
    \end{align}
    implying the $D_{12}$ identity.     Lastly for $D_{22}$ identity, we integrate by parts as,
    \begin{align}
        &r_{2} \int \exp(2Kr_{1} \cos\theta + 2Kmr_{2} \cos2\theta) \ d\theta  \nonumber\\
        &= \int \cos2\theta \ \exp(2Kr_{1} \cos\theta + 2Kmr_{2} \cos2\theta) \ d\theta \nonumber\\
        &=- \int \frac{\sin 2\theta}{2} (-2Kr_{1} \sin \theta - 4Kmr_{2} \sin 2\theta) \ \exp(2Kr_{1} \cos\theta + 2Kmr_{2} \cos2\theta) \ d\theta.
    \end{align}
    Using the identities $\sin \theta \sin 2\theta = \cos \theta - \cos \theta \cos 2\theta$ and $\sin^{2}2\theta = 1 - \cos^{2} 2\theta,$ we deduce \begin{align}\label{eq:r2D22}
        r_{2} = Kr_{1} ( r_{1} - D_{12}) + 2Kmr_{2} (1- D_{22}).
    \end{align}
    Substituting $D_{12}$ into this equation yields the desired formula for $D_{22}$. 
    This establishes \eqref{D formulas}.

    Turning to the identities \eqref{IBP first}-\eqref{IBP third}, 
    the first two equations follow from \eqref{eq:r1D12}, \eqref{eq:r2D22}, respectively. For the third, we note that the first two may be rewritten as
    \begin{align}
        \begin{bmatrix}
            1-2K(1-D_{11}) & -4Km(r_{1} - D_{12})\\
            -K(r_{1}-D_{12}) & 1-2Km(1-D_{22})
        \end{bmatrix}
        \begin{bmatrix}
        r_{1} \\
        r_{2}
        \end{bmatrix}
        =\begin{bmatrix}
        0 \\
        0
        \end{bmatrix}.
    \end{align}
    Since $(r_1, r_2) \neq (0, 0),$ the coefficient matrix must be singular. Computing its determinant then yields \eqref{IBP third}. This completes the proof of the lemma.

\subsection{Proof of \cref{decomposition}}\label{ssec:stab_od}
We now turn to the proof of \cref{decomposition}.

Since $J$ only has non-zero first and second Fourier coefficients, it is immediate that $q(J *u) = 0$ for $u\in V_0$. We can then write the decomposition \begin{align}
        L^{2}_{1/q} 
        &= V_{0} \oplus \text{span}_{\mathbb{R}} \{ q(\theta) \sin\theta, q(\theta) \sin 2\theta\} \oplus \text{span}_{\mathbb{R}}\{ q(\theta) \cos\theta, q(\theta) \cos 2\theta \}.
    \end{align}
    The orthogonality of these subspaces follows from the direct computation. We analyze the decompositions of $\text{span}_{\mathbb{R}} \{ q(\theta) \sin\theta, q(\theta) \sin 2\theta\}$ and $ \text{span}_{\mathbb{R}}\{ q(\theta) \cos\theta, q(\theta) \cos 2\theta \}$ separately.

   Identifying $u= q \big( a\sin\theta+ b \sin 2\theta \big)$ with $(a, b) \in \bR^2$,
   \begin{align}
       J\ast u = \begin{bmatrix}
            K(1-D_{11}) & K(r_1 - D_{12})\\
            Km(r_1 - D_{12}) & Km(1-D_{22})
        \end{bmatrix}\begin{bmatrix}
            a \\ b
        \end{bmatrix}.
   \end{align}
    Solving for eigenvalues, we find 
    \begin{equation}
    \begin{aligned}
        \tilde{\mu} = {K \over 2} \Big( (1-D_{11}) + m (1-D_{22}) + \sqrt{ [(1-D_{11}) - m (1-D_{22})]^{2} + 4m(r_1 - D_{12})^{2} } \Big), \\
        \mu = {K \over 2} \Big( (1-D_{11}) + m (1-D_{22}) - \sqrt{ [(1-D_{11}) - m (1-D_{22})]^{2} + 4m(r_1 - D_{12})^{2} } \Big),
    \end{aligned}
    \end{equation}
    with respective eigenvectors $\begin{bmatrix}
            r_1\\
            2mr_2
        \end{bmatrix},
        \begin{bmatrix}
            -2r_2 \\
            r_1
        \end{bmatrix}$.
    Using \cref{D expression lemma}\eqref{IBP third}, one checks that $\tilde{\mu} = 1/2.$ 
    Since ${1/ 2}\ne \mu$, their corresponding eigenspaces $V_{1/2}$ and $V_{\mu}$ are orthogonal.

    Similarly, for $u = aq\cos\theta + bq\cos(2\theta)$, we can write
    \begin{align}
    J\ast u =     \begin{bmatrix}
            KD_{11} & KD_{12}\\
            KmD_{12} & KmD_{22}
        \end{bmatrix}
        \begin{bmatrix}
            a \\ b
        \end{bmatrix}.
    \end{align}
    The eigenvalues are $\lambda_{\pm}$ as in \eqref{lambda def}, and the corresponding eigenvectors are  $\begin{bmatrix}
            \eta_{1}^{j }\\
            \eta_{2}^{j}
        \end{bmatrix}$
    for $j \in \{\pm\}$ and $\eta_i^{\pm}$ as in \eqref{eta def}. Applying all decompositions, we complete the proof.

\subsection{Proof of \cref{opt in V0}}\label{ssec:opt_in_V0}
Next, we give the proof of \cref{opt in V0}.

Since  the domain $C\coloneqq \{u\in V_{0}  : \int u = 0\}$ is convex, existence of a unique minimizer $\hat{u} \in V_{0}$ follows from the Hilbert projection theorem. Recalling \cref{decomposition}, the orthogonality relations are
\begin{align}
        0=\int \hat{u}(\theta)\cos(j\theta) \ d\theta = \int \hat{u}(\theta) \sin(j\theta) \ d\theta, \qquad {1}\le j\le 2.
    \end{align}
    By the method of Lagrange multipliers, there exists $a, b, c, d, e \in \bR$ such that \begin{align}
        \hat{u}(\theta) = -1 + a  q(\theta) + b  q(\theta) \cos \theta + c  q(\theta) \cos 2\theta+d  q(\theta) \sin \theta+ e  q(\theta) \sin 2\theta.
    \end{align}
    The relations $\int \hat{u} \sin \theta = \int \hat{u} \sin 2\theta=0$ lead to the linear system
    \begin{align}
        \begin{bmatrix}
            \int q(\theta) \sin^{2} \theta \ d\theta & \int q(\theta) \sin \theta \sin 2\theta \ d\theta\\
            \int q(\theta) \sin \theta \sin 2\theta \ d\theta & \int q(\theta) \sin^{2} 2\theta  \ d\theta\\
        \end{bmatrix}
        \begin{bmatrix}
            d\\
            e
        \end{bmatrix}
        =
        \begin{bmatrix}
            0\\
            0
        \end{bmatrix}.
    \end{align}
    Since the determinant of the left-hand matrix is positive, it follows that $d = e =0.$

    The remaining relations $\int \hat{u}  = \int \hat{u} \cos \theta = \int \hat{u} \cos 2\theta=0$  yield
    \begin{align}
        \begin{bmatrix}
            1 & \int q(\theta) \cos \theta \ d\theta & \int q(\theta) \cos 2\theta \ d\theta\\
            \int q(\theta) \cos \theta \ d\theta & \int q(\theta) \cos^{2}\theta \ d\theta & \int q(\theta) \cos \theta \cos 2\theta \ d\theta\\
            \int q(\theta) \cos 2\theta \ d\theta & \int q(\theta) \cos\theta \cos 2\theta \ d\theta & \int q(\theta) \cos^{2} 2\theta \ d\theta
        \end{bmatrix}
        \begin{bmatrix}
            a\\
            b\\
            c
        \end{bmatrix}
        =2\pi
        \begin{bmatrix}
            1\\
            0\\
            0
        \end{bmatrix}.
    \end{align}
    The first leading principal minor of the left matrix is $1,$ and the second leading principal minor is 
    \begin{align}
        \int q(\theta) \cos^{2} \theta \ d\theta - \Big(\int q(\theta) \cos \theta \ d\theta\Big)^{2} >0.
    \end{align}
    The determinant of the full matrix is 
        \begin{equation}
        \begin{aligned}
            &\left(\int q(\theta) \cos^{2} \theta \ d\theta - \Big(\int q(\theta) \cos \theta \ d\theta \Big)^{2} \right) \left(\int q(\theta) \cos^{2} 2\theta \ d\theta- \Big(\int q(\theta) \cos 2\theta \ d\theta\Big)^{2}\right)\\
            &\quad - \left(\int q(\theta) \cos \theta \cos 2\theta \ d\theta - \Big(\int q(\theta) \cos \theta \ d\theta\Big)\Big(\int q(\theta) \cos 2\theta \ d\theta\Big)\right)^{2}\\
            &= \Var_{\theta \sim q}[\cos \theta]  \Var_{\theta \sim q}[\cos 2\theta] - (\cov_{\theta \sim q}[\cos \theta, \cos2\theta])^{2} >0.
        \end{aligned}
        \end{equation}
    Solving for $a, b, c$ gives
    \begin{align}
        \begin{bmatrix}
            a\\
            b\\
            c
        \end{bmatrix}
        = (2\pi) 
         \begin{bmatrix}
            1 & r_{1} & r_{2}\\
            r_{1} & D_{11} & D_{12}\\
            r_{2} & D_{12} & D_{22}
        \end{bmatrix}^{-1}
        \begin{bmatrix}
            1\\
            0\\
            0
        \end{bmatrix}.
    \end{align}
   Substituting $\hat{u} = -1 + a q(\theta) + b q(\theta) \cos \theta+ c q(\theta)\cos 2\theta,$ we calculate
    \begin{align}
        \|1+\hat{u}\|_{L^{2}_{1/q}}^2 &=\begin{bmatrix}
            a & b & c
        \end{bmatrix}
       \begin{bmatrix}
            1 & r_{1} & r_{2}\\
            r_{1} & D_{11} & D_{12}\\
            r_{2} & D_{12} & D_{22}
        \end{bmatrix}
        \begin{bmatrix}
            a\\
            b\\
            c
        \end{bmatrix} \nonumber\\
        &=(2\pi)^{2} \frac{\bE_{\theta \sim q}[\cos^{2}\theta]\bE_{\theta \sim q}[\cos^{2}2\theta] - (\bE_{\theta \sim q}[\cos\theta \cos2\theta])^{2}}{\Var_{\theta \sim q}[\cos \theta]  \Var_{\theta \sim q}[\cos 2\theta] - (\cov_{\theta \sim q}[\cos \theta, \cos2\theta])^{2}}.
    \end{align}
    {\eqref{eq:opt in V0} now follows from $D_{ij} = \bE_{\theta\sim q}[\cos(i\theta)\cos(j\theta)]$ and $\cov_{\theta\sim q}[\cos(i\theta),\cos(j\theta)] = D_{ij}-r_ir_j$. This completes the proof of the proposition.}

\subsection{Proof of \cref{Dirichlet lower}}\label{ssec:stab_LB}
In this subsection, we consider the proof of \cref{Dirichlet lower}. We have already sketched most of it in \cref{ssec:stab_out}. What remains for us to show are (1)  the positive definiteness of the matrix $\fC$ defined in \eqref{eq:C++def}-\eqref{eq:C--def} and (2) the conclusion of the lower bound \eqref{eq:Dirichlet lower} from the positive definiteness of $\fC$. Throughout this subsection, we use the same notation as in \cref{ssec:stab_out}.


    Starting with item (1), the positive definiteness of $\fC$ relies on the following lemma, the proof of which is deferred to \cref{ssec:tech_maindet}.

    \begin{lemma}\label{lem:maindet}
    For $K>1$ and $m \in [0,m_{0}]$, we have
    \begin{multline}\label{Main Det}
    (D_{11}D_{22} - D_{12}^{2})  \left[ (r_{1} \eta_{1}^{+} + r_{2} \eta_{2}^{+})^{2} 
    \left( \frac{1}{2} - \lambda_{-} \right) \alpha_{\lambda_{-}} + (r_{1} \eta^{-}_{1} + r_{2} \eta^{-}_{2})^{2} 
    \left( \frac{1}{2} - \lambda_{+} \right) \alpha_{\lambda_{+}} \right] \\
    +2\big((D_{11} - r_{1}^{2})(D_{22} - r_{2}^{2}) - (D_{12} - r_{1}r_{2})^{2}\big) \left( \frac{1}{2} - \lambda_{+} \right) \left( \frac{1}{2} - \lambda_{-} \right) 
    \alpha_{\lambda_{+}}\alpha_{\lambda_{-}} >0.
    \end{multline}
\end{lemma}

{The reader may check that the left-hand side of \eqref{Main Det} equals
\begin{align}
    2\Big((D_{11}-r_1^2)(D_{22}-r_2^2) - (D_{12}-r_1r_2)^2\Big) \det \fC.
\end{align}
Note that the prefactor is positive by using the covariance representation of $D_{ij}-r_ir_j$ and Cauchy-Schwarz, and therefore $\det \fC >0$ by \cref{lem:maindet}. To see now that $\fC$ is positive definite, note that $\fC^{--}>0$ since $\lambda_{-}<1/2$. Positive definiteness now follows from Sylvester's criterion.}

For item (2), which will conclude the proof of the proposition, let $\gamma>0$ denote the smallest eigenvalue of $\fC$. Then
\begin{align}
    &\frac{a_0^2}{2}\|1+\hat{v}_0\|_{L_{1/q}^2}^2 + \Big(\frac12-\mu\Big)\|u_\mu\|_{L_{1/q}^2}^2 + \Big(\frac12-\lambda_+\Big)\|u_{\lambda_+}\|_{L_{1/q}^2}^2 + \Big(\frac12-\lambda_-\Big)\|u_{\lambda_-}\|_{L_{1/q}^2}^2 \nonumber\\
    &=\begin{bmatrix}
            c_{\lambda_{+}} & c_{\lambda_{-}} 
        \end{bmatrix}
        \fC
        \begin{bmatrix}
            c_{\lambda_{+}}\\
            c_{\lambda_{-}} 
        \end{bmatrix}
        + \Big({1 \over 2} - \mu \Big) \alpha_{\mu} c_{\mu}^{2} \nonumber\\
   &\ge \frac{\gamma}{\alpha_{\lambda_+}}\|u_{\lambda_+}\|_{L_{1/q}^2}^2 + \frac{\gamma}{\alpha_{\lambda_-}}\|u_{\lambda_-}\|_{L_{1/q}^2}^2 + \Big(\frac12-\mu\Big)\|u_{\mu}\|_{L_{1/q}^2}^2.\label{eq:Dlb1}
\end{align}
Note that it is tautological that $\alpha_{\lambda_\pm}>0$ (recall  \eqref{eq:alladef}). We cannot conclude the proof just yet because the last line is missing the term $a_0^2\|1+\hat{v}_0\|_{L_{1/q}^2}^2$. However, recalling the expression \eqref{eq:a0} for $a_0$ and \eqref{eq:opt in V0} for $\|1+\hat{v}_0\|_{L_{1/q}^2}^2$, we see that
\begin{align}
    &a_0^2\|1+\hat{v}_0\|_{L_{1/q}^2}^2 \nonumber\\
    &=  \frac{D_{11}D_{22} - D_{12}^{2}}{(D_{11} - r_{1}^{2})(D_{22} - r_{2}^{2}) -(D_{12} - r_{1}r_{2})^{2}}  \big((r_{1} \eta_{1}^{+} + r_{2} \eta_{2}^{+}) c_{\lambda_{+}}+(r_{1} \eta^{-}_{1} + r_{2} \eta^{-}_{2}) c_{\lambda_{-}}\big)^{2} \nonumber\\
    &\le \frac{2(D_{11}D_{22} - D_{12}^{2})}{(D_{11} - r_{1}^{2})(D_{22} - r_{2}^{2}) -(D_{12} - r_{1}r_{2})^{2}}\big((r_{1} \eta_{1}^{+} + r_{2} \eta_{2}^{+})^2 c_{\lambda_{+}}^2 + (r_{1} \eta_{1}^{-} + r_{2} \eta_{2}^{-})^2 c_{\lambda_{-}}^2\big) \nonumber\\
    &\le \frac{2(D_{11}D_{22} - D_{12}^{2})}{(D_{11} - r_{1}^{2})(D_{22} - r_{2}^{2}) -(D_{12} - r_{1}r_{2})^{2}}\Big( \frac{(r_{1} \eta_{1}^{+} + r_{2} \eta_{2}^{+})^2}{\gamma} \frac{\gamma}{\alpha_{\lambda_+}}\|u_{\lambda_{+}}\|_{L_{1/q}^2}^2 \nonumber\\
    &\phantom{=}\qquad+\frac{(r_{1} \eta_{1}^{-} + r_{2} \eta_{2}^{-})^2}{\gamma}\frac{\gamma}{\alpha_{\lambda_-}}\|u_{\lambda_{-}}\|_{L_{1/q}^2}^2\Big) ,\label{eq:Dlb2}
\end{align}
where we have used the elementary inequality $(x+y)^{2} \leq 2 (x^{2}+y^{2})$. Combining the inequalities \eqref{eq:Dlb1}, \eqref{eq:Dlb2}, we find that
\begin{multline}
    \frac{a_0^2}{2}\|1+\hat{v}_0\|_{L_{1/q}^2}^2 + \Big(\frac12-\mu\Big)\|u_\mu\|_{L_{1/q}^2}^2 + \Big(\frac12-\lambda_+\Big)\|u_{\lambda_+}\|_{L_{1/q}^2}^2 + \Big(\frac12-\lambda_-\Big)\|u_{\lambda_-}\|_{L_{1/q}^2}^2  \\
    \ge C_0\Big(a_0^2\|1+\hat{v}_0\|_{L_{1/q}^2}^2 + \|u_\mu\|_{L_{1/q}^2}^2  + \|u_{\lambda_+}\|_{L_{1/q}^2}^2 + \|u_{\lambda_-}\|_{L_{1/q}^2}^2\Big),
\end{multline}
where 
\begin{align}\label{eq:C0def}
    C_0 \coloneqq \min\Bigg(\frac12-\mu, {\frac{\gamma}{\frac{2(D_{11} D_{22} - D_{12}^{2})}{(D_{11}-r_1^2)(D_{22}-r_2^2)-(D_{12}-r_1r_2)^2 }\max(r_1\eta_1^+ + r_2\eta_2^+, r_1\eta_1^- + r_2\eta_2^-)^{2} + \min(\alpha_{\lambda_+}^{-1}, \alpha_{\lambda_-}^{-1})}}\Bigg).
\end{align}
Applying this lower bound to the first line on the right-hand side of the relation \eqref{eq:Drewrite}, we conclude that
\begin{align}
    \mathbf{D}(u) &\ge   C_0\Big(a_0^2\|1+\hat{v}_0\|_{L_{1/q}^2}^2 + \|u_\mu\|_{L_{1/q}^2}^2  + \|u_{\lambda_+}\|_{L_{1/q}^2}^2 + \|u_{\lambda_-}\|_{L_{1/q}^2}^2\Big)\nonumber\\
    &\phantom{=} + \frac{a_0^2}{2}\Big(\|1+v_0\|_{L_{1/q}^2}^2 - \|1+\hat{v}_0\|_{L_{1/q}^2}^2\Big) \nonumber\\
    &\ge \min\Big(C_0,\frac12\Big)\|u-u_{1/2}\|_{L_{1/q}^2}^2.
\end{align}
This completes the proof of \cref{Dirichlet lower}.

\subsection{Conclusion of the proof}\label{ssec:stab_conc}

We now have all the necessary components to conclude the proof of \cref{stability thrm}, which follows from a Poincar\'{e} inequality.

Set $e\coloneqq \frac{q'}{ \| q' \|_{H^{-1}_{1/q}}}$. By orthogonal projection, observe that
\begin{align}
    \|u-\inner{u}{e}_{H_{1/q}^{-1}}e\|_{H_{1/q}^{-1}} = \inf_{c\in\bR} \|u-cq'\|_{H_{1/q}^{-1}} &= \inf_{c\in\bR} \|\p_{\theta}^{-1}(u-cq')\|_{L_{1/q}^2} \nonumber\\
    &\le \sqrt{\frac{\max q}{\min q}}\inf_{c\in\bR} \|u-cq'\|_{L_{1/q}^2} \nonumber\\
    &=\sqrt{\frac{\max q}{\min q}}\|u-u_{1/2}\|_{L_{1/q}^2}. \label{eq:Poincare}
\end{align}
where we have used Plancherel's theorem, that $V_{1/2} = \mathrm{span}\{q'\}$, and implicitly that $u$ and $q'$ have zero mean. Combining the estimate \eqref{eq:Poincare} with \cref{Dirichlet lower} and performing a little bookkeeping, we see that the proof of \cref{stability thrm} is complete.

\section{Technical estimates}\label{sec:Tech}
This section is devoted to the technical estimates used in the previous section, in particular the proof of \cref{lem:maindet}, which is crucial to proving the lower bound for the Dirichlet form. As previously commented, the proof of \cref{lem:maindet} relies on several, rather involved auxiliary approximations and estimates that are split across the following subsections.

The remainder of the section is organized as follows. In \cref{ssec:Tech_var}, we prove some  bounds for the covariance of Fourier coefficients of $q$. In \cref{estimation of r1 and r2}, we show first-order asymptotics for $r_1,r_2$ with respect to the parameter $m$, which will allow us to approximate various quantities in terms of their $m=0$ values when $m$ is sufficiently small. In \cref{ssec:Tech_eval}, we prove some elementary bounds for the eigenvalues $\lambda_{\pm}, \mu$ of the operator $qJ\ast(\cdot)$ (recall their definitions from \eqref{lambda def}, \eqref{mu def}, respectively). In \cref{ssec:Tech_omni}, we state and prove an omnibus lemma (\cref{{technical inequalities}}) that will allow us to estimate individual quantities appearing in the proof of \cref{lem:maindet}. Finally, in \cref{ssec:tech_maindet}, we combine all of the preceeding individual results to prove \cref{lem:maindet}.

 \subsection{Covariance bounds}\label{ssec:Tech_var}

The following lemma establishes key inequalities for the expressions $D_{ij} - r_{i}r_{j}$, which are consequences of their definitions \eqref{eq:DbarDdef}.  

\begin{lemma} \label{E inequalities lemma}
    Let $K>1$. Then for $j=1,2$,
        \begin{align}\label{E positive} 
            0< D_{ii} - r_i^2  < 1 \quad \text{and} \quad -1+\indic_{m\le \frac14} < D_{ij} - r_ir_j.
        \end{align}
        In particular, $D_{ij}- r_i r_j>0$ if $m\le {1/4}$. Furthermore,        
        \begin{align}\label{E inequalities}
            D_{11}-r_{1}^{2} \leq {1 \over 2K} - Kr_{1}(D_{12}-r_{1}r_{2}), \quad
            D_{22}-r_{2}^{2} \leq \frac{1}{K} -\frac{D_{12}-r_{1}r_{2}}{Kr_{1}}.
        \end{align}

\end{lemma}
\begin{proof}
   We assume that $m>0$ so that in our application of \cref{I technical lemmas}, the variable $y>0$. The case $m=0$ follows from letting $m \to 0^+$.
   
   For \eqref{E positive}, we write $D_{ij}-r_{i}r_{j} = \cov_{\theta \sim q}[\cos(i\theta), \cos(j\theta)]$, which implies that $D_{ii}-r_i^2\in (0,1)$ and $D_{ij}-r_{i}r_j \in (-1,1)$. 
    Applying \cref{I technical lemmas}\eqref{I property 1} with  $x=2Kr_1$ and $y=2Kmr_2$ completes the argument.\footnote{Here (and only here), we  use the assumption that $m\le 1/4$ to know that $r_2>0$ in order to show that $D_{12}-r_1r_2>0$.}
    
    For \eqref{E inequalities}, we apply \eqref{I property 1} with $x=2Kr_1$, $y=2Kmr_2$  to obtain
    \begin{align}
       \frac{1}{2K^{2}r_{1}} \left(1 - 2K(D_{11} - r_{1}^{2}) \right) \ge D_{12}-r_{1}r_{2}.
    \end{align}
    Rearranging yields the first inequality of \eqref{E inequalities}. The second inequality follows from applying \eqref{I property 2} with $x=2Kr_{1}$, $y = 2Kmr_{2}$. 
\end{proof}

\subsection{Approximation for $r_{j}$} \label{estimation of r1 and r2}
We now establish approximations for $r_1,r_2$ in terms of $\bar{r}_{1},\bar{r}_{2}$, the latter properties being relatively well studied.

By \cite{bertini}, it is known that
\begin{align}\label{Kuramoto relation}
   \bar{r}_2 = 1- {1 \over K} < \bar{r}_{1}^{2} < 1 - {1 \over 2K} = \bar{D}_{11}.
\end{align}
The lower bound is sharp as $K \to 1,$ while the upper bound is sharp as $K \to \infty.$ The next lemma strengthens the lower bound for large $K$ and the upper bound for $K$ near $1.$ 

\begin{lemma} \label{Kuramoto better estimates}
    For $K>1,$ we have \begin{align} \label{Kuramoto estimate}
        1 - {1 \over 2K} - {1 \over 2K^{2}} < \bar{r}_{1}^{2} < 1 - {1 \over 2K} \quad \text{and} \quad 1 - {1 \over K} < \bar{r}_{1}^{2} < 2 \left(1 - {1 \over K}\right).
    \end{align}
    In particular, $\frac{\bar{r}_{1}^{2}}{\bar{r}_{2}} \approx 1$.
\end{lemma}
\begin{proof}
    For the first inequality, consider the identity \begin{align}
        \bar{r}_{1} \bar{D}_{12} - \bar{r}_{2} \bar{D}_{11} &= \frac{1}{2I_{0}^{2}} \left[ (I_{1}^{2} - I_{0} I_{2}) -  (I_{2}^{2} - I_{1} I_{3})\right] (2K\bar{r}_{1}) =\Big(\frac{ T_{1} - T_{2}}{2I_{0}^{2}}\Big) (2K\bar{r}_{1}),
    \end{align}
    where $T_{n}$ is the Tur\'anian introduced in \cref{turanian lemma}. Hence, $\bar{r}_{1} \bar{D}_{12} - \bar{r}_{2} \bar{D}_{11} > 0.$ Rewriting this expression in $\bar{r}_{1},\bar{r}_{2}$ using the identities $ \bar{D}_{11} = \frac{1+ \bar{r}_{2}}{2}$, $\bar{D}_{12} = \bar{r}_{1} - \frac{\bar{r}_{2}}{K\bar{r}_{1}}$,  we obtain
    \begin{align}
        \bar{r}_{1} \bar{D}_{12} - \bar{r}_{2} \bar{D}_{11} = \bar{r}_{1}^{2} - \bar{r}_{2} \left(1 + \frac{1}{2K} \right) = \bar{r}_{1}^{2} - \left(1 - {1 \over 2K} - {1 \over 2K^{2}}\right),
    \end{align}
    where the final equality holds by $\bar{r}_{2} = 1-{1 / K}.$     This establishes the desired lower bound.

    For the second inequality, we only need to prove the upper bound. This follows from the Tur\'anian inequality \cref{turanian lemma}\eqref{turanian 2} with $x=2K\bar{r}_{1}$.  
\end{proof}

For our model, we seek inequalities similar to those in \cref{Kuramoto better estimates} for $r_{1},r_{2}$ with general $m.$ However, analyzing the solutions to the two-parameter self-consistency equation directly is challenging. Instead, we approximate $r_1,r_2$ as linear perturbations of $\bar{r}_1,\bar{r}_{2}$ in terms of $m$ and leverage the known properties of $\bar{r}_1,\bar{r}_2$. {The restriction $m\le {1/4}$ is to ensure the quantity $D_{12}-r_1r_2>0$ as in \cref{E inequalities lemma}\eqref{E positive}. We also need to ensure strict convexity of the free energy at $(r_1,r_2)$, but this holds for $m\le {1/2}$.} 

\begin{prop}
Let $K>1$ and $m \in {[0, 1/4]}$. The functions $r_{j}$ are strictly increasing in $m$. In particular, $r_{j} \geq \bar{r}_{j}$.     
Furthermore, we have the bounds
\begin{align}
    &D_{12} < r_{1}, \label{D12 and r1 relation} \\
    &r_{2} \leq \bar{r}_{2} + \frac{4\bar{r}_{2}}{(1-2m)^2K} m, \label{r2 upper}\\
    &r_{1} \leq \bar{r}_{1} + \frac{\bar{r}_{2}}{(1-2m)^{2}K\bar{r}_{1}} m \leq  \bar{r}_{1} + \frac{4\bar{r}_{2}}{K\bar{r}_{1}} m . \label{r1 upper}
\end{align}
\end{prop}
\begin{proof}
    The $m=0$ case may be directly check using \eqref{Kuramoto estimate}. We assume {$m \in (0, 1/4]$} and first prove the monotonicity. Fix $K>1.$ Recall that $(r_{1}, r_{2})$ is a  critical point in $(0, 1)^{2}$ {(unique if $m\le 1/4$)} of the free energy $F(r_1,r_2;m)$ introduced in \eqref{F def}.
     Using the implicit function theorem, the derivative with respect to $m$ of the vector-valued function $(r_1,r_2)$ is  \begin{align}\label{eq:pmr1r2}
        \left[{\partial r_{1} \over \partial m}, \ {\partial r_{2} \over \partial m} \right]^{\intercal} = - \left( \nabla^{2} F(r_1, r_2; m) \right)^{-1} {\partial \over \partial m} \nabla F(r_1, r_2; m).
    \end{align}
    The Hessian of the free energy at $(r_1, r_2)$ is  \begin{align}
        \nabla^{2} F(r_1, r_2; m) = 2K 
        \begin{bmatrix}
            1-2K(D_{11} - r_{1}^{2}) & - 2Km (D_{12} - r_{1}r_{2}) \\
            - 2Km (D_{12} - r_{1}r_{2}) & m(1-2Km (D_{22} - r_{2}^{2}))
        \end{bmatrix},
    \end{align}
which is positive definite by \cref{positive hessian in first}. This implies, again by Sylvester's criterion and positive determinant, that 
    \begin{equation}
    \begin{aligned}
        &1 - 2K(D_{11} - r_{1}^{2}) >0, \\
        &(1-2K(D_{11} - r_{1}^{2}))(1-2Km(D_{22} - r_{2}^{2})) > 4K^{2}m (D_{12} - r_{1} r_{2})^{2}, \\
        &1-2Km(D_{22} - r_{2}^{2}) >0.
    \end{aligned}
    \end{equation}
    {Since $D_{12}-r_1r_2>0$ by \cref{E inequalities lemma}\eqref{E positive}, it follows that all entries of $(\nabla^{2} F(r_1, r_2; m))^{-1}$ are positive with }
    \begin{align}
         \left(\nabla^{2} F(r_1, r_2; m)\right)^{-1} = \frac{2K}{\det \nabla^{2} F(r_1, r_2; m)} 
        \begin{bmatrix}
            m(1-2Km(D_{22} - r_{2}^{2})) &  2Km (D_{12}-r_{1}r_{2})\\
            2Km (D_{12} - r_{1}r_{2}) & 1-2K(D_{11} - r_{1}^{2})
        \end{bmatrix}.
    \end{align}
    On the other hand, 
    \begin{align}
        \frac{\partial}{\partial m} \nabla F(r_1, r_2; m) = \begin{bmatrix}
            -4K^{2}r_2 (D_{12} - r_{1}r_{2})\\
            -4K^{2}mr_2(D_{22}-r_{2}^{2})
        \end{bmatrix}.
    \end{align}
    Substituting this and the inverse Hessian matrix into the formula \eqref{eq:pmr1r2} for $ \big({\partial r_{1} \over \partial m}, \frac{\partial r_2}{\partial m} \big)$, we observe that both $\frac{\partial r_i}{\partial m}>0$. This proves that $r_1,r_2$ are strictly increasing in $m$.

    \medskip
    To prove \eqref{r2 upper}, we consider the third and fourth Fourier coefficients 
    \begin{equation}
    \begin{aligned}
        {I_{3}(2Kr_1, 2Kmr_2) \over I_{0}(2Kr_1, 2Kmr_2)} &= \int \cos 3\theta \ q(\theta) d\theta = r_{1} \left(1 - \frac{r_2 - \bar{r}_{2}}{2mr_{2}} \right),\\ {I_{4}(2Kr_1, 2Kmr_2) \over I_{0}(2Kr_1, 2Kmr_2)} &= \int \cos 4\theta \ q(\theta) d\theta=1 + \frac{r_{1}^{2} (r_2 - \bar{r}_{2})}{4m^{2} r_{2}^{2}} - \frac{1}{Km}.
    \end{aligned}
    \end{equation}
    From $I_{3}/I_{0} >0$, we obtain
    \begin{align}\label{temp r_2}
        r_{2} < \frac{\bar{r}_{2}}{1-2m}.
    \end{align}
    Moreover $I_{4}/ I_{0} < 1$ leads to
    \begin{align}\label{temp r_2'}
        r_{2} - \bar{r}_{2} < \frac{4r_{2}^{2}}{Kr_{1}^{2}} \ m
         < \frac{4}{K(1-2m)^{2}} \ \frac{\bar{r}_{2}^{2}}{\bar{r}_{1}^{2}} \ m \leq \frac{4m\bar{r}_2}{K(1-2m)^{2}},
    \end{align}
    where the second inequality is by \eqref{temp r_2} and the third by \eqref{Kuramoto relation}. Writing $r_2 = \bar{r}_2 + (r_2-\bar{r}_2)$ and combining \eqref{temp r_2}, \eqref{temp r_2'} yields the desired estimate.

    \medskip

    Lastly, to prove \eqref{r1 upper}, we observe that \begin{align}
        \frac{\partial}{\partial m} r_{1} &= \frac{2Kr_{2} (D_{12} - r_{1}r_{2})}{(1-2Km(D_{22} - r_{2}^{2}))(1-2K(D_{11}-r_{1}^{2})) - 4K^{2}m(D_{12} - r_{1}r_{2})^{2}} \nonumber\\
        &< \frac{1}{K(1-2m)} \frac{r_2}{r_1} \nonumber\\
        &\leq {\frac{1}{K(1-2m)^{2}}}{\bar{r}_{2} \over \bar{r}_{1}} 
    \end{align}
    where the first inequality follows from \eqref{E inequalities},  and the second inequality from \eqref{temp r_2} and $\bar{r}_{2} \leq \bar{r}_{2}^{1/2} \leq \bar{r}_{1}$. The mean-value theorem implies that there exists $m_{0} \in [0, m]$ such that \begin{align}
        r_{1} - \bar{r}_{1} = {\partial r_{1} \over \partial m} (K, m_0)  m < \frac{\bar{r}_{2}}{(1-2m)^{2} K \bar{r}_{1}} m.
    \end{align}
    This completes the proof.
\end{proof}

\subsection{Eigenvalue estimates}\label{ssec:Tech_eval}
To investigate the contribution of $m$ to various quantities in the sequel, it will be convenient to introduce the notation
\begin{align}\label{eq:Deltadef}
        \Delta\coloneqq D_{11} - \sqrt{(D_{11} - m D_{22})^{2} + 4m D_{12}^{2}},
\end{align}
which obviously depends on $K$ and $m$.

With this notation, we can re-express the eigenvalues $\lambda_{\pm}, \eta_2^{\pm}$ of the operator $qJ\ast(\cdot)$ (recall their definitions from \eqref{lambda def}, \eqref{eta def})
\begin{equation}
\begin{aligned}
    \lambda_{+}&= K \left( D_{11} + \frac{mD_{22} - \Delta}{2} \right), \quad \lambda_{-} =  \frac{K(mD_{22} + \Delta)}{2},\\
    \eta_{2}^{+} &= \frac{mD_{22}-\Delta}{2}, \quad \eta^{-}_{2} = -D_{11} + \frac{mD_{22} + \Delta}{2}.
\end{aligned}
\end{equation}
We previously observed that $\eta_1^{\pm} = D_{12}$, which is sufficient as is.

The next lemma provides quantitative estimates of $mD_{22}+\Delta$ and $mD_{22} - \Delta$ as functions of $m$.

\begin{lemma}\label{lem:Delta estimates}
    For $K>1$ and $m \in [0, 1/4],$  the following estimates hold,
    \begin{align} \label{Delta estimates}
         0\leq mD_{22} - \Delta \leq 8mD_{12}^{2} \quad \text{and} \quad 0 \leq mD_{22} +\Delta \leq \frac{3m(D_{11} D_{22} - D_{12}^{2})}{D_{11}}.
    \end{align}
\end{lemma}
\begin{proof}
By Cauchy-Schwarz, $D_{11} D_{22} - D_{12}^{2}>0$. Writing
\begin{align}
    (D_{11}-mD_{22})^2 + 4mD_{12}^2 = (D_{11}+mD_{22})^2 -4m(D_{11}D_{22}-D_{12}^2), 
\end{align}
it follows from their definitions \eqref{lambda def}, \eqref{eta def} that $\lambda_{-}, \eta_{2}^{+} \ge 0$. Hence, $mD_{22} \pm \Delta \ge 0$, taking care of the lower bound in \eqref{Delta estimates}. We now show the upper bounds.

\medskip
Unpacking the definition \eqref{eq:Deltadef} of $\Delta$, we compute
    \begin{align}
        mD_{22} - \Delta &= -(D_{11} - m D_{22}) + \sqrt{(D_{11} - m D_{22})^{2} + 4m D_{12}^{2}} \nonumber\\
        &=\frac{4m D_{12}^{2}}{(D_{11} - m D_{22}) + \sqrt{(D_{11} - m D_{22})^{2} + 4m D_{12}^{2}}}.
    \end{align}
    Since $m \leq 1/4,$ we crudely lower bound $D_{11} - mD_{22} \geq 1/2 - m \geq 1/4$ to obtain
    \begin{align}
        mD_{22} - \Delta \leq \frac{4mD_{12}^{2}}{2 ( D_{11} - mD_{22})} \leq 8mD_{12}^{2},
    \end{align}
    which gives the first upper bound in \eqref{Delta estimates}.

    For the second upper bound, observe that
    \begin{align}
        \sqrt{(D_{11}-mD_{22})^{2} + 4m D_{12}^{2}}\geq \sqrt{D_{11}^{2} -2m(D_{11}D_{22} - 2D_{12}^{2})} = D_{11} \sqrt{1-m \delta},
    \end{align}
    where $\delta \coloneqq {2(D_{11}D_{22} - 2D_{12}^{2})}/ {D_{11}^{2}}.$
    If $m \leq 1/4$, {since $D_{22} \leq 1, \ D_{11} = (1+r_2)/2 \geq 1/2,$ then}
    \begin{align}
        m \delta <  \frac{2mD_{22}}{D_{11}} \leq  4m \leq 1,
    \end{align}
    which ensures that $1-m \delta > 0$. Since $\sqrt{1-m\delta} \geq 1- m\delta$, 
    \begin{align}
        \sqrt{(D_{11}-mD_{22})^{2} + 4m D_{12}^{2}} \geq D_{11} - \frac{2m(D_{11}D_{22} - 2D_{12}^{2})}{D_{11}}.
    \end{align}
    This leads to
    \begin{align}
        mD_{22} + \Delta = mD_{22} + D_{11} -\sqrt{(D_{11}-mD_{22})^{2} + 4m D_{12}^{2}} \leq \frac{3m(D_{11} D_{22} - D_{12}^{2})}{D_{11}},
    \end{align}    
    as claimed.
\end{proof}

Using \cref{lem:Delta estimates}, we obtain the following estimates for the eigenvalues of the operator $qJ\ast(\cdot)$. As commented in \cref{ssec:stab_out}, when $m=0$, $\lambda_+ = K-1/2$ and $\lambda_-=\mu=0$.

\begin{lemma}\label{lem:evalests}
    Let $K>1$. If  $m \in (0, 1/4]$,
    \begin{equation} \label{lambda}
    \begin{aligned}
            \lambda_{+} > {K-{1 \over 2}}, \quad 0< \mu < {1 \over 2}, \quad \text{and} \quad {0< \lambda_{-} <\frac{81m}{4} }.
    \end{aligned}
    \end{equation}
\end{lemma}
\begin{proof}
    Suppose $m \in (0, 1/4]$. From the definitions of $D_{ij}$ and Cauchy-Schwarz, 
    \begin{align}
        (1-D_{11})(1-D_{22}) - (r_1 - D_{12})^{2} =\int \sin^{2} \theta dq \int \sin^{2} 2\theta dq  - \Big( \int \sin\theta \sin 2\theta dq \Big)^{2} >0.
    \end{align}
    It follows that 
    \begin{equation}
    \begin{aligned}
        &D_{11} + m D_{22} > \sqrt{(D_{11} - m D_{22})^{2} + 4m D_{12}^{2}},\\
        &(1-D_{11}) + m (1-D_{22}) > \sqrt{\big((1-D_{11}) - m (1-D_{22})\big)^{2} + 4m(r_{1} - D_{12})^{2}}.
    \end{aligned}
    \end{equation}
    These results imply that $\lambda_{+}>\lambda_{-}>0$ and ${1/ 2} > \mu >0.$
    
    Since $mD_{22} - \Delta >0$, we have  \begin{align}
        \lambda_{+} 
        &= K \Big( D_{11} + \frac{ m D_{22} - \Delta}{2} \Big) 
        > K D_{11}.
    \end{align}
    Using the monotonicity of $r_{2}$ in $m$ and the fact $D_{11} = (1+r_2)/2$, we obtain the bound \begin{align}
        D_{11} > \bar{D}_{11} = 1 - {1 \over 2K}.
    \end{align}
Inserting above, it follows that $\lambda_{+} > K(1 - {(2K)^{-1}}) = K - 1/2$.

    For the remaining upper bound for $\lambda_{-}$, 
    by \eqref{Delta estimates}, 
    \begin{align}
        \lambda_{-} = {K \over 2}  (mD_{22} + \Delta) < \frac{3mK(D_{11}D_{22} - D_{12}^{2})}{2D_{11}}.
    \end{align}
   The bounds of \cref{E inequalities lemma} imply
   \begin{align}
        D_{11}< r_{1}^{2} + {1 \over 2K}, \quad D_{22} < r_{2}^{2} + {1 \over K}, \quad D_{12} > r_{1}r_{2}.
    \end{align}
    Applying these 
    \begin{align} 
        K \left( D_{22} - \frac{D_{12}^{2}}{D_{11}} \right) < K\left(r_{2}^{2} + {1 \over K} - \frac{r_{1}^{2} r_{2}^{2}} {r_{1}^{2} + {1 \over 2K}} \right) = 1 + \frac{r_{2}^{2}}{2r_{1}^{2} + {1 \over K}} <1 + \frac{(1+16m)^{2}} {2}  \frac{\bar{r}_{2}^{2}}{\bar{r}_{1}^{2}} \leq {27 \over 2},
    \end{align}
    from which the upper bound in \eqref{lambda} is immediate.
    
\end{proof}

\subsection{Omnibus lemma}\label{ssec:Tech_omni}
The following omnibus lemma contains a list of estimates that will be used in the proof of \cref{lem:maindet}. For the statement of the lemma, the reader will recall the definitions \eqref{eq:al12def}, \eqref{eq:almudef}, \eqref{eq:alladef} of the $L_{1/q}^2$ norms denoted by $\alpha_{1/2},\alpha_{\mu},\alpha_{\lambda_\pm}$.  By straightforward calculations, these quantities may be expressed as follows (we only need to consider $\alpha_{\lambda_{\pm}}$ in the ensuing calculations):
\begin{align}
    \label{eq:alla+} &\alpha_{\lambda_{+}}=D_{12}^{2}D_{11} + D_{12} (mD_{22} - \Delta) + {D_{22} \over 4} (mD_{22} - \Delta)^{2},\\
    \label{alpha_2 expression} &\alpha_{\lambda_{-}}=D_{11} (D_{11}D_{22} - D_{12}^{2}) - (D_{11}D_{22} - D_{12}^{2}) (mD_{22} +\Delta) + {D_{22} \over 4} (mD_{22} + \Delta)^{2}.
\end{align}

\begin{lemma} \label{technical inequalities}
    Suppose $K>1$ and $m \in [0, 1/4]$.
    \begin{enumerate}[(1)]
        \item\label{item:first}
        We have
        \begin{align} \label{first}
            r_{1}\eta_{1}^{+} + r_{2} \eta_{2}^{+} \geq r_{1} D_{12}.
        \end{align}
        
        \item\label{item:C1 inequality}
        For $C_{1} = 81/2$,
        \begin{align}\label{C1 inequality}
            {1 \over 2} - \lambda_{-} \geq { 1- C_{1} m \over 2}.
        \end{align}
        
        \item\label{item:C2 inequality}
        For $C_2 = 3$,
        \begin{align}\label{C2 inequality}
             (1- C_{2} m) D_{11} (D_{11} D_{22} - D_{12}^{2}) \leq \alpha_{\lambda_{-}} \leq D_{11} (D_{11} D_{22} - D_{12}^{2}).
        \end{align}
        
        \item\label{item:C3 inequality}
        For $C_{3} = 6$,
        \begin{align} \label{C3 inequality}
            D_{11} + \eta^{-}_{2} \leq C_{3} m  D_{11} (D_{11}D_{22} - D_{12}^{2}).
        \end{align}

        \item\label{item:C4 inequality}
        For $C_{4} = 28$,
        \begin{align}\label{C4 inequality}
            {1 \over 2} - \lambda_{+} \geq - (1+C_4 m)(K-1).
        \end{align}

        \item\label{item:C5 inequality}
        For $C_5 = 32$,
        \begin{align}\label{C5 inequality}
            \alpha_{\lambda_{+}} \leq (1+C_5 m) D_{12}^{2} D_{11}.
        \end{align}
       
        \item\label{item:final estimates}
        We have
        \begin{align} \label{final estimate}
            \Big(K-{1 \over 2}\Big) r_{1}^{2} - (K-1) D_{11} >0.
        \end{align}
        
        \item\label{item:C6 inequality}
        If {$m \in [0, {1 /32})$}, then for $C_{6} =8$,
        \begin{align} \label{C6 inequality}
            r_{1}^{2} <{\frac{C_6(1+4m+8m^2)}{1-32m}}\left( \Big(K-{1 \over 2}\Big) r_{1}^{2} - (K-1) D_{11} \right).
        \end{align}
        
        \item\label{item:C7 and C8 inequalities}
        For $C_{7} = 9/2$ and $C_{8} = 5/2$,
        \begin{align} \label{C7 and C8 inequalities}
            (K-1)  (r_{1} D_{12} - r_{2} D_{11}) < C_{7},\quad
            (K-1) (D_{11} D_{22} - D_{12}^{2}) < C_{8}. 
        \end{align}
        
        \item\label{item:C9 inequality}
        {For $C_9 = {3915/4}$,}
        \begin{align} \label{C9 inequality}
            (K-1) \big(r_{2}^{2} (D_{11} - \eta^{-}_{2}) - 2r_{1}r_{2} D_{12}\big) > - C_{9}   \left( \Big(K-{1 \over 2}\Big) r_{1}^{2} - (K-1) D_{11} \right).
        \end{align}
    \end{enumerate}
\end{lemma}
\begin{proof}
    
    The assertion \ref{item:first} follows directly from $\eta_{2}^{+} >0$ and $\eta_1^+ = D_{12}$ (recall \eqref{eta def}), while \ref{item:C1 inequality} follows from \eqref{lambda}.
        
    For \ref{item:C2 inequality}, we observe that \eqref{alpha_2 expression} implies
    \begin{align}
            \alpha_{\lambda_{-}} < D_{11} (D_{11} D_{22} - D_{12}^{2}) \iff {D_{22} \over 4} (mD_{22} + \Delta) < D_{11} D_{22} - D_{12}^{2}.
    \end{align}
        Applying the bound \eqref{Delta estimates}, we find 
        \begin{align}
            {D_{22} \over 4} (mD_{22} + \Delta) < \frac{3m D_{22}}{4D_{11}} (D_{11} D_{22} - D_{12}^{2}) < D_{11} D_{22} - D_{12}^{2},
        \end{align}
        where the last inequality follows from $m<2/3,$ and $D_{11} \geq {1 / 2}, \ D_{22} \leq 1$. This takes care of the upper bound. For the lower bound, using (\ref{alpha_2 expression}) and (\ref{Delta estimates}), we have
        \begin{align}
        \nonumber\alpha_{\lambda_{-}}
            & > D_{11} (D_{11}D_{22} - D_{12}^{2}) - (D_{11}D_{22} - D_{12}^{2}) (mD_{22} +\Delta) \\
            \label{tempor}& > D_{11} (D_{11}D_{22} - D_{12}^{2}) - 3m  \frac{(D_{11}D_{22} - D_{12}^{2})^{2}}{D_{11}}.
        \end{align}
        By the Gr\"uss inequality (e.g. see \cite[Theorem 1.1]{Gruss}), $D_{11}D_{22} - D_{12}^{2} \leq 1/4,$ which implies \begin{align}
            \frac{D_{11}D_{22} - D_{12}^{2}}{D_{11}} \leq {1 \over 2} \leq D_{11}.
        \end{align}
        Inserting this estimate to \eqref{tempor} completes the argument.
        
        Turning to \ref{item:C3 inequality}, recalling the definition \eqref{eta def} of $\eta^{-}_{2}$,
        \begin{align}
            D_{11} + \eta^{-}_{2} < {3m \over 2}  \frac{D_{11}D_{22} - D_{12}^{2}}{D_{11}} \leq 6m  D_{11} (D_{11}D_{22} - D_{12}^{2}),
        \end{align}
        where the last inequality follows from $D_{11} \geq 1/2.$
        
        For \ref{item:C4 inequality}, unpacking the definition \eqref{lambda def} of $\lambda_{+}$,
        \begin{align}
            {1 \over 2} - \lambda_{+} = {1 \over 2} - K D_{11} - {K \over 2} (mD_{22} -\Delta).
        \end{align}
        Recalling that $D_{11} =(1+r_{2})/2$, we obtain from the bound \eqref{r2 upper},
        \begin{align}
            KD_{11}  \leq {K \over 2} \Big(1+\bar{r}_{2} + {16m \over K} \bar{r}_{2} \Big) = \Big(K-{1\over2}\Big) + 8m \bar{r}_{2}= \Big(K-{1 \over 2}\Big) + \frac{8m(K-1)}{K},
        \end{align}
        where we also use $\bar{r}_{2} = 1 - 1/K$. Furthermore,
        \begin{align}
            {K \over 2} (mD_{22} - \Delta) &< 4m K D_{12}^{2} < 4m  Kr_{1}^{2} \nonumber\\
            & \leq 4m  K \Big(\bar{r}_{1}^{2} + {8\bar{r}_{2} \over K} m + \frac{16 \bar{r}_{2}^{2}}{K^{2} \bar{r}_{1}^{2}}  m^{2}\Big)\nonumber\\
            & \leq 4m  K (2\bar{r}_{2} + {2\bar{r}_{2}} + { \bar{r}_{2}}) = 20m (K-1),
        \end{align}
        where the first line follows from \eqref{Delta estimates}, \eqref{D12 and r1 relation}; the second line from \eqref{r1 upper}; and the last from \eqref{Kuramoto estimate}.
        Recalling our starting point, we finish the proof.
        
        For \ref{item:C5 inequality}, we claim that
        \begin{align}\label{eq:C5ine0}
            D_{12} (mD_{22} -\Delta) \geq {D_{22} \over 4} (mD_{22} - \Delta)^{2} \iff mD_{22} - \Delta \leq {4D_{12} \over D_{22}}.
        \end{align} 
        To see this, observe that \eqref{Delta estimates} and $D_{12}, D_{22} \leq 1$ imply that
        \begin{align}
            mD_{22} - \Delta < 8m D_{12}^{2} \leq 8m D_{12} \leq {4 \over D_{22}} D_{12}.
        \end{align}
        Applying the bound \eqref{eq:C5ine0} to the third term in the expression \eqref{eq:alla+} for $\alpha_{\lambda_+}$, it follows that  \begin{align}
            \alpha_{\lambda_{+}} \leq D_{12}^{2}D_{11} + 2D_{12} (mD_{22} - \Delta) <  D_{12}^{2}D_{11} + 16m D_{12}^{3} \leq (1+32m)D_{12}^{2} D_{11},
        \end{align}
        where the second inequality follows from \eqref{Delta estimates} again and the third from $D_{12} \leq 1, \ D_{11} \geq 1/2.$
        
        Considering \ref{item:final estimates}, we apply \eqref{E inequalities} to obtain
        \begin{align}
            \Big(K-{1\over2}\Big) r_{1}^{2} - (K-1) D_{11} = {r_{1}^{2} \over 2} - (K-1) (D_{11}-r_{1}^{2})> {r_{1}^{2} - \bar{r}_{2} \over 2} > {\bar{r}_{1}^{2} - \bar{r}_{2} \over 2}\geq 0.
        \end{align}
        
       Now for \ref{item:C6 inequality}, assuming $m < {1/32}$, we first verify the inequality \eqref{C6 inequality} when $m=0.$ Using \eqref{Kuramoto estimate}, we have \begin{align}
            \bar{r}_{1}^{2} > \bar{r}_{2} \Big(1 + {1 \over 2K} \Big).
        \end{align} From $\bar{D}_{11} \geq 1/2,$ we write\begin{align}
            \Big( K - {1 \over 2} \Big) - (K-1) \bar{D}_{11} = K \bar{D}_{11}  (\bar{r}_{1}^{2} - \bar{r}_{2}) > {\bar{r}_{2} \over 4}.
        \end{align} Using $\bar{r}_{1}^{2} \leq \bar{r}_{2}$, we derive the estimates as, \begin{align} \label{C6 estimates}
            \bar{r}_{1}^{2} < 8  \left( \Big(K-{1 \over 2}\Big) \bar{r}_{1}^{2} - (K-1) \bar{D}_{11} \right),\quad 
            \bar{r}_{2} < 4  \left( \Big(K-{1 \over 2}\Big) \bar{r}_{1}^{2} - (K-1) \bar{D}_{11} \right)
        \end{align}Furthermore, we establish 
        \begin{equation} \label{intermediate C6}
        \begin{aligned}
            \Big( K - {1 \over 2} \Big) r_{1}^{2} - (K-1) D_{11} &> \Big( K - {1 \over 2} \Big) \bar{r}_{1}^{2} - (K-1) \Big( \bar{D}_{11} + {8m \over K} \bar{r}_{2} \Big)\\
            & > (1-32m)  \left( \Big( K - {1 \over 2} \Big) \bar{r}_{1}^{2} - (K-1) \bar{D}_{11} \right).
        \end{aligned}
        \end{equation}
        Similarly, we derive an upper bound for $r_1^{2}$ as, \begin{align}
            r_{1}^{2} < \Big( \bar{r}_{1} + {4\bar{r}_{2} \over K\bar{r}_{1}}m \Big)^{2} &= \bar{r}_{1}^{2} + \frac{8\bar{r}_{2}}{K}m + \frac{16\bar{r}_{2}^{2}}{K^{2}\bar{r}_{1}^{2}}m^{2} \nonumber\\
            & < (8 +32m+64m^{2}) \left( \Big( K - {1 \over 2} \Big) \bar{r}_{1}^{2} - (K-1) \bar{D}_{11} \right),
        \end{align}
        where we use \eqref{r1 upper} and \eqref{C6 estimates}. Combining this with \eqref{intermediate C6}, we obtain the desired result.

        For \ref{item:C7 and C8 inequalities}, to prove the first inequality, we upper-bound $r_{1}D_{12} - r_{2} D_{11}.$ If $m>0,$ then Lemma \ref{D expression lemma} and \eqref{r1 upper} imply \begin{align}
            r_{1} D_{12} - r_{2} D_{11} &= r_{1}^{2} - \frac{r_{1}^{2} (r_2 - \bar{r}_{2})}{4mr_{2}} - \frac{r_{2} + r_{2}^{2}}{2} < r_{1}^{2} - r_{2}^{2} \nonumber\\
            &< \Big(\bar{r}_{1} + \frac{4\bar{r}_{2}}{K \bar{r}_{1}} m \Big)^{2} - \bar{r}_{2}^{2} \leq \frac{3}{2K} + \frac{8}{K} m + {16 \over K} m^{2},
        \end{align}
        where we use \begin{align}
            \bar{r}_{1}^{2} - \bar{r}_{2}^{2} \leq \Big(1-{1 \over 2K}\Big) - \Big(1-{1 \over K}\Big)^{2} < {3 \over 2K}
        \end{align}
        for the last inequality. Thus, we obtain the first upper bound \begin{align}
            (K-1) (r_{1} D_{12} - r_{2} D_{11}) < {3 \over 2} + 8m + 16m^{2} \leq {9 \over 2}.
        \end{align}
        And $m=0$ case can be done by taking $m \to 0+$. For the second inequality, we estimate $D_{11}D_{22} - D_{12}^{2}$ as follows, \begin{align}
            D_{11} D_{22} - D_{12}^{2} &< D_{11} - D_{12}^{2} = (D_{11} - r_{1}^{2}) + r_{1}^{2} - D_{12}^{2} \nonumber\\
            & \leq {1 \over 2K} + r_{1}^{2} - r_{1}^{2}r_{2}^{2} <  {1 \over 2K} + 1 - r_{2}^{2} \leq {1 \over 2K} + 1 - \bar{r}_{2}^{2} < {5 \over 2K}, 
        \end{align}
        where the second inequality is from Lemma \ref{E inequalities lemma}.
        Therefore, we have \begin{align}
            (K-1) (D_{11} D_{22} - D_{12}^{2}) < {5 \over 2}.
        \end{align}
        
        Lastly, for \ref{item:C9 inequality}, we note
        \begin{align}
            2r_{1}r_{2} D_{12} - r_{2}^{2} (D_{11} - \eta^{-}_{2}) &= 2 r_{2} (r_{1} D_{12} - r_{2} D_{11}) + \frac{r_{2}^{2} (mD_{22}+\Delta)}{2} \nonumber\\
            &\leq 2 r_{2} (r_{1} D_{12} - r_{2} D_{11}) + \frac{3mr_{2}^{2}}{2} \frac{D_{11}D_{22} - D_{12}^{2}}{D_{11}},
        \end{align}
        {where the inequality follows from \eqref{Delta estimates}.}
        From this bound, we deduce
        \begin{align}
            (K-1) \big(2r_{1}r_{2} D_{12} - r_{2}^{2} (D_{11} - \eta^{-}_{2})\big) &\leq 2r_{2} (K-1)(r_{1}D_{12} - r_{2} D_{11}) + \frac{3mr_{2}^{2}}{2D_{11}} (K-1)(D_{11}D_{22} - D_{12}^{2}) \nonumber\\
            &< 9r_{2} + \frac{15m}{4D_{11}} r_{2}^{2} \nonumber\\
            &\leq \Big(9+{15m \over 2} \Big)(1+16m) \bar{r}_{2} \nonumber\\
            &< 18\Big(9+{15m \over 2} \Big)(1+16m) \left( \Big(K-{1\over2}\Big) r_{1}^{2} - (K-1) D_{11} \right),
        \end{align}
        where the second inequality is from \eqref{C7 and C8 inequalities}, and the last from \eqref{C6 inequality}. {Using that $m\le 1/4$ to bound the prefactors on the last line completes the proof.}
\end{proof}


\medskip 


\subsection{Proof of \cref{lem:maindet}}\label{ssec:tech_maindet}
We now have all the necessary ingredients to prove \cref{lem:maindet}, which we do in this subsection.

    We note that \eqref{lambda}, \eqref{lambda} imply
    \begin{align}
        {1 \over 2} - \lambda_{+} < 0 < {1 \over 2} - \lambda_{-}.
    \end{align}
    Let us denote the left-handed side of \eqref{Main Det} by $\mathbb{D}$. From \eqref{first}, \eqref{C1 inequality}, \eqref{C2 inequality}, $\eta_{1}^{-} = D_{12},$ \eqref{C4 inequality}, \eqref{C5 inequality}, we obtain
    \begin{multline}
        \mathbb{D}
        \geq (D_{11}D_{22} - D_{12}^{2})  \Bigl[(1-C_{1}m)(1-C_{2}m)  \frac{r_{1}^{2}}{2}  D_{12}^{2} D_{11}(D_{11}D_{22} - D_{12}^{2})\\
         -(1+C_{4}m)(1+C_{5}m)  (K-1) (r_{1} D_{12} + r_{2} \eta^{-}_{2})^{2} D_{12}^{2} D_{11} \Bigr]\\
        -(1+C_{4}m)(1+C_{5}m)  (K-1)\big((D_{11} - r_{1}^{2})(D_{22} - r_{2}^{2}) - (D_{12} - r_{1}r_{2})^{2}\big) D_{12}^{2} D_{11}^{2} (D_{11} D_{22} - D_{12}^{2}). 
    \end{multline}
   Note that all the terms  on the right-hand side of the inequality have a factor of $D_{12}^2D_{11}$. Pulling out the factor $(D_{11}D_{22} - D_{12}^{2}) D_{12}^{2} D_{11}$ from the right-hand side leads to the expression
   \begin{multline}\label{eq:maindet0}
       (1-C_{1}m)(1-C_{2}m)   \frac{r_{1}^{2}}{2}  (D_{11}D_{22} - D_{12}^{2}) -(1+C_{4}m)(1+C_{5}m)  (K-1) (r_{1} D_{12} + r_{2} \eta^{-}_{2})^{2}  \\
         -(1+C_{4}m)(1+C_{5}m)  (K-1)\big((D_{11} - r_{1}^{2})(D_{22} - r_{2}^{2}) - (D_{12} - r_{1}r_{2})^{2}\big)  D_{11}.
   \end{multline}
   Expanding $(r_1D_{12}+r_2\eta_2^{-})^2$ and $(D_{11}-r_1^2)(D_{22}-r_2^2) - (D_{12}-r_1r_2)^2$, then regrouping terms, we see that
   \begin{multline}\label{eq:maindet1}
       (\ref{eq:maindet0}) = (D_{11} D_{22} - D_{12}^{2}) \Big((1-C_{1}m)(1-C_{2}m)  \frac{r_{1}^{2}}{2} - (1+C_{4}m)(1+C_{5}m)  (K-1) D_{11}\Big)\\
        +(1+C_{4}m)(1+C_{5}m)  (K-1) r_{1}^{2}(D_{11} D_{22} -D_{12}^{2})  \\
        +(1+C_{4}m)(1+C_{5}m)  (K-1)(D_{11} + \eta^{-}_{2})(r_{2}^{2} (D_{11} - \eta_{2}^{-}) - 2r_{1}r_{2}D_{12}).
   \end{multline}
    Combining the bounds \eqref{C3 inequality}, \eqref{C9 inequality} from
    \cref{technical inequalities},
    \begin{multline}\label{eq:maindet1'}
        (K-1)(D_{11} + \eta^{-}_{2})(r_{2}^{2} (D_{11} - \eta_{2}^{-}) - 2r_{1}r_{2}D_{12}) \\
        \ge -C_3C_9m D_{11}(D_{11}D_{22}-D_{12}^2)\Big((K-\frac12) - (K-1)D_{11}\Big).
    \end{multline}
    Applying this bound to the last line of \eqref{eq:maindet1} and combining the first two lines, we obtain
    \begin{multline}\label{eq:maindet2}
    (\ref{eq:maindet0}) \ge  (1+C_{4}m)(1+C_{5}m) \left( \Big(K - {1 \over 2} \Big) r_{1}^{2} - (K-1) D_{11} \right)\\
        -\big( (1+C_{4}m)(1+C_{5}m) -(1-C_{1}m)(1-C_{2}m) \big)  \frac{r_{1}^{2}}{2}\\
        -C_{3}C_{9}m(1+C_{4}m)(1+C_{5}m)   D_{11}  \left( \Big(K-{1 \over 2}\Big) r_{1}^{2} - (K-1) D_{11} \right).
    \end{multline}
    We want all the terms on the right-hand side to have a factor $(K-1/2)r_1^2 - (K-1)D_{11} >0$. To accomplish this, we apply \eqref{C6 inequality} from \cref{technical inequalities} to the $r_1^2$ factor. Upon pulling out the common factor, we arrive at the expression

   \begin{multline}
    (1+C_{4}m)(1+C_{5}m) -C_{3}C_{9}m(1+C_{4}m)(1+C_{5}m)   D_{11}  \\
    -\big( (1+C_{4}m)(1+C_{5}m) -(1-C_{1}m)(1-C_{2}m) \big)  {\frac{C_6(1+4m+8m^2)}{2(1-32m)}} \\
    > (1+C_{4}m)(1+C_{5}m) -C_{3}C_{9}m(1+C_{4}m)(1+C_{5}m) \\
    -\big( (1+C_{4}m)(1+C_{5}m) -(1-C_{1}m)(1-C_{2}m) \big)  {\frac{C_6(1+4m+8m^2)}{2(1-32m)}} ,
   \end{multline}
   where the inequality is since $D_{11}<1$.  Noting that the preceding  expression is positive for  $m \in [0, m_0]$, with $m_0$ defined as in \eqref{eq:m0_def}, completes the proof.

\bibliographystyle{alpha}
\bibliography{references,references_Kyunghoo}

\end{document}